\documentclass{amsart}
\usepackage{graphicx}
\usepackage{amsthm}
\usepackage{amsmath}
\usepackage{amssymb}
\usepackage{amscd}

\newtheorem{theorem}{Theorem}[section]
\newtheorem{lemma}[theorem]{Lemma}
\newtheorem{proposition}[theorem]{Proposition}
\newtheorem{corollary}[theorem]{Corollary}

\theoremstyle{definition}

\theoremstyle{remark}
\newtheorem{remark}[theorem]{Remark}
\newtheorem*{remarks}{Remarks}

\numberwithin{equation}{section}

\newcommand\Nn{\mathbb{N}}

\newcommand\Zz{\mathbb{Z}}
\newcommand\Rr{\mathbb{R}}
\newcommand{\Rrm}[1]{{\mathbb{M}^{#1}}}
\newcommand{\Rrv}[1]{\mathbb{R}^{#1}}
\newcommand{\SO}[1]{\mathop{SO}(#1)}
\newcommand{\id}{\operatorname{Id}}
\newcommand{\tensor}[1]{\mathbb{#1}}
\newcommand{\tsym}{\mathbb{T}_{\sym}}
\DeclareMathOperator{\grad}{\nabla\!}
\DeclareMathOperator{\ud}{\, \mathrm{d}}
\DeclareMathOperator{\dist}{\operatorname{dist}}
\DeclareMathOperator{\sym}{\operatorname{sym}}
\DeclareMathOperator{\asym}{\operatorname{skw}}
\newcommand{\esssup}{\mathop{\operatorname{ess\,sup}}}
\DeclareMathOperator{\trace}{\operatorname{tr}}
\newcommand{\transpose}{^\textrm{\footnotesize T}}
\newcommand{\norm}[1]{\left\Vert #1\right\Vert}
\newcommand{\abs}[1]{\left\vert #1\right\vert}
\newcommand{\iprod}[2]{\left\langle#1,\,#2\right\rangle}
\newcommand{\dprod}[2]{\left\langle#1,\,#2\right\rangle}
\newcommand{\fs}{}
\newcommand\conv{\rightarrow}
\newcommand\wconv{{\rightharpoonup}}
\newcommand\gammaconv{\stackrel{\Gamma}{\relbar\!\rightarrow}}
\newcommand{\per}{{\operatorname{per}}}
\newcommand{\lin}{{\operatorname{lin}}}
\renewcommand{\hom}{{\operatorname{hom}}}
\newcommand{\super}[1]{^{(#1)}}
\newcommand{\mc}{\super{\operatorname{mc}}_\hom}
\newcommand{\boundaryset}{\mathcal{A}_\gamma}
\newcommand\restterm{\operatorname{r}}
\newcommand{\energy}[1]{\mathcal{#1}}
\newcommand{\mtext}[1]{\quad\text{ #1 }\quad}
\newcommand{\step}[1]{\textit{\underline{Step #1.}}}

\begin{document}
\title[On the Commutability of Homogenization and Linearization]{On the Commutability of Homogenization and Linearization in Finite Elasticity}
\author[S. M\"uller \& S. Neukamm]{Stefan M\"uller \& Stefan Neukamm}
\address[Stefan M\"uller]{Hausdorff Center for Mathematics,  Institute for Applied Mathematics, Universit\"at Bonn,  Endenicher Allee 60, 53115 Bonn, Germay}
\email{stefan.mueller@hcm.uni-bonn.de}
\address[Stefan Neukamm]{Zentrum Mathematik / M6, Technische Universit\"{a}t M\"{u}nchen,
Boltzmannstra\ss e 3,  85748 Garching bei M\"{u}nchen, Germany}
\curraddr{Max Planck Institute for Mathematics in the Sciences,
  Inselstra\ss e 22, 04103 Leipzig, Germany}
\email{neukamm@mis.mpg.de}

\date{\today}

\begin{abstract}
    We consider a family of non-convex integral functionals
    \begin{equation*}
      \frac{1}{h^2}\int_\Omega W(x/\varepsilon,\id+h\grad g(x))\ud x,\qquad g\in\fs W^{1,p}(\Omega;\Rrv n)
    \end{equation*}
    where $W$ is a Carath\'eodory function periodic in its first, and non-degenerate in its second variable. We prove under
    suitable conditions that the $\Gamma$-limits corresponding to linearization
    ($h\rightarrow 0$) and homogenization ($\varepsilon\rightarrow 0$)
    commute, provided $W$ is minimal at the identity and admits
      a quadratic Taylor expansion at the identity. Moreover, we show that the homogenized integrand, which is
    determined by a multi-cell homogenization formula, has a quadratic Taylor
    expansion with a quadratic term that is given by the homogenization of
    the second variation of $W$.

\vspace{10pt}

\vspace{6pt}

\end{abstract}

\maketitle

\section{Introduction}
In this contribution we consider periodic integral functionals of the type
\begin{equation}\label{intro:0}
  \int\limits_\Omega W(x/\varepsilon,\grad u(x))\ud x,\qquad u\in\fs W^{1,p}(\Omega;\Rrv n)
\end{equation}
where $\varepsilon$ is a positive parameter, $\Omega$ an open, bounded Lipschitz domain in $\Rrv n$ with
$n\geq 2$ and $W:\,\Rrv n{\times}\Rrm n\rightarrow [0,+\infty)$ is a
Carath\'eodory function $Y:=[0,1)^n$-periodic in its first variable.

Functionals of this type model manifold situations in physics and
engineering. We are particularly interested in applications to
elasticity. In this context the integral in \eqref{intro:0} is the
elastic energy of a periodic composite material with period
$\varepsilon$ deformed by the map $u:\Omega\rightarrow\Rrv n$. In
situations where the deformation is close to a rigid deformation, say
$\abs{\grad u-\id}\sim h$, it is
natural to introduce the scaled displacement ${g(x):=h^{-1}(u(x)-x)}$
and to consider the scaled, but equivalent energy
\begin{equation*}
  \mathcal I^{\varepsilon,h}(g):=\frac{1}{h^2}\int\limits_\Omega
  W(x/\varepsilon,\id + h \grad g(x))\ud x,\qquad g\in\fs
  W^{1,p}(\Omega;\Rrv n).
\end{equation*}
For small parameters
$\varepsilon$ and $h$ we expect that the energy $\mathcal I^{\varepsilon,h}$ can be replaced by an
\textit{effective} model that is simpler than the original one, but
nevertheless captures the essential behavior from a macroscopic perspective. In this context the limit $h\rightarrow 0$
corresponds to \textit{linearization}, while the theory of \textit{homogenization}
studies the asymptotic behavior of the energy as
$\varepsilon\rightarrow 0$. Obviously, there are two orderings to derive an
effective model by consecutively passing to the limits
$\varepsilon\rightarrow0 $ and $h\rightarrow 0$. The goal of this paper is to show
that both ways lead to the same limiting energy. In other words we
prove that \textit{homogenization and linearization commute}.

It is well-known that in the case where the integrand $W$ is convex in its second variable the
\textit{homogenization} of \eqref{intro:0}  is the integral functional $\int_\Omega W_\hom\super 1(\grad u(x))\ud x$, where the homogenized integrand is given by the \textit{one-cell homogenization} formula
\begin{equation*}
  W_\hom\super 1(F):=\inf\left\{\,\int\limits_Y W(y, F+\grad\varphi(y))\ud y\,:\,\varphi\in\fs W^{1,p}_\per(Y;\Rrv n)\,\right\}.
\end{equation*}
This result goes back to P.~Marcellini \cite{Marcellini1978} and was
extensively studied with various methods (cf. e.g. \cite{DalMaso},
\cite{Allaire1992}). The more general setting of monotone operators
and non-periodic dependency on the spatial variable was studied by L.~Tartar \cite{Tartar1977,Tartar2009,TartarBook2009}. On the other side, for non-convex $W$ it turns
out that in general the relaxation of $W$ over one --- or even an ensemble of
finitely many copies of the periodicity cell $Y$ is not
sufficient for homogenization. Nevertheless, A.~Braides
\cite{Braides1985,Braides1998} and the first author \cite{Mueller1987} showed in the
1980s that in the non-convex case the \textit{homogenization} of the
functional in \eqref{intro:0}  leads to the \textit{multi-cell homogenization} formula
\begin{equation*}
  W\super{\text{mc}}_\hom(F):=\inf\limits_{k\in\Nn}\inf\left\{\,\frac{1}{k^n}\int\limits_{kY} W(y,F{+}\grad \varphi(y))\ud y\,:\,\varphi\in\fs W^{1,p}_\per(kY;\Rrv n)\,\right\}.
\end{equation*}
We state this result in a precise manner in Theorem~\ref{thm:homog}
below.  The fact that the one-cell formula is not sufficient in
nonlinear elasticity has been observed earlier in the engineering
literature (see \cite{Triantafyllidis1984,Triantafyllidis1985}). These
authors observed that instabilities arise even when the one-cell
formula predicts no loss of stability. Later in \cite{Mueller1993}
this has been systematically investigated.

In this contribution we are interested in integrands of the following
type: For $a>0$ and $p\in(1,\infty)$ let $\mathcal W(a,p)$ denote the
class of Carath\'eodory functions $W:\Rrv n{\times}\Rrm n\rightarrow
[0,\infty)$ that are $Y$-periodic in the first variable and that
satisfy the following conditions (W1) -- (W4).
We suppose that $W$ is locally Lipschitz continuous and
  satisfies standard growth conditions of order $p$:
\begin{equation}
  \tag{W1}
  \left\{\begin{aligned}
      &\tfrac{1}{a}\abs{ F}^p-a\leq W(y, F)\leq a(1+\abs{ F}^p)\mtext{and}\\
      &\abs{W(y,F)-W(y,G)}\leq a(1+\abs{F}^{p-1}+\abs{G}^{p-1})\abs{F-G}
    \end{aligned}\right.
\end{equation}
for all $F,G\in\Rrm n$ and a.e. $y\in\Rrv n$.
We suppose that $F=\id$ is a natural state and $W$ is non-degenerate, i.e.
\begin{align*}
  \tag{W2}&W(y,\id)=0\qquad\text{for a.e. }y\in\Rrv n,\\
  \tag{W3}&W(y, F)\geq \tfrac{1}{a}\dist^2( F,\SO n)\qquad\text{for all
    $F\in\Rrm n$ and a.e. $y\in\Rrv n$.}
\end{align*}
Furthermore, we suppose that $W$ admits a quadratic Taylor expansion at $\id$ in the sense that
\begin{equation*}
  \tag{W4}  \exists\,Q\in\mathcal Q\,:\,\limsup\limits_{
    G\rightarrow 0\atop G\neq 0}\esssup\limits_{y\in \Rrv n}\frac{\abs{W(y,\id+ G)-Q(y,G)}}{\abs{ G}^2}=0.
\end{equation*}
Here and below $\mathcal Q$ denotes the set of Carath\'eodory functions\linebreak
$Q:\,\Rrv n{\times}\Rrm n\rightarrow [0,+\infty)$ that are
$Y$-periodic in the first variable, non-negative and quadratic in the second variable and
bounded in the sense that
\begin{equation*}
  \forall G\in\Rrm n\,:\,\esssup_{y\in\Rrv n}Q(y,G)\leq c\abs{G}^2
\end{equation*}
for a suitable constant $c>0$.

\begin{remarks}
  \begin{enumerate}
  \item In elasticity it is natural to assume that $W$ is frame
    indifferent, i.e.
    \begin{equation*}
      W(y, R F)=W(y, F)\qquad\text{for all $R\in\SO n$,
        $F\in\Rrm n$ and a.e. $y\in\Rrv n$.}
    \end{equation*}
    We have not introduced this assumption explicitly, because it is
    not required in the proof. Nevertheless, the lower bound (W3) is
    motivated by frame indifference.
  \item (W2) \& (W3) are the main assumptions in our analysis.
      For energy densities representing composite materials, condition
      (W2) requires that each of the single components of the
      composite is stress free in a common reference configuration.
      Note that this rules out the application to prestressed
      composites, as considered in Section~\ref{sec:count-ii:-non}.
  \item The combination of (W3) \& (W4) can be interpreted as a
    generalization of Hooke's law to the geometrically nonlinear
    setting, in the sense that both conditions together imply that for
    infinitesimal small strains the stress is proportional to the
    strain.
  \end{enumerate}
\end{remarks}

For ${W\in\mathcal W(a,p)}$ we study the behavior of the homogenized
integrand $W\mc$ near the identity. Our first main result is the following: 
\begin{theorem}\label{thm:1}
  Let $W\in\mathcal W(a,p)$. Then 
  \begin{equation}\label{thm:1:1}
    \limsup\limits_{ G\rightarrow 0,\atop G\neq 0}\frac{\abs{W\mc(\id+
        G)-Q\super 1_\hom(G)}}{\abs{ G}^2}=0.
  \end{equation}
  The statement remains valid if we drop condition (W1).
\end{theorem}
This means that whenever the integrand $W$ admits a quadratic
expansion at $\id$ with a quadratic term $Q$,
then also the homogenized integrand $W\mc$ has a quadratic expansion
where the quadratic term is given by the homogenization of $Q$.
Partial results in this direction under strong implicit assumptions on
the minimizers for the cell problems have been obtained in \cite{Mueller1993}.

\begin{remark}\label{rem:1}
Conditions (W3) \& (W4) imply that the quadratic integrand $Q(y,\cdot)$ is a positive semi-definite quadratic form for almost
  every $y\in\Rrv n$. More precisely, the non-degeneracy condition (W3) implies that $Q(y,\cdot)$ restricted to the subspace of symmetric $n{\times}n$ matrices is positive definite; in particular, we have
  \begin{equation*}
    Q(y, G)\geq \tfrac{1}{a}\abs{\sym G}^2\qquad\text{for all $G\in\Rrm n$
      and a.e. $y\in\Rrv n$.}
  \end{equation*}
  In virtue of Korn's inequality this guarantees that  the minimization problem
  \begin{equation*}
    \int\limits_Y
    Q(y, G+\grad\varphi(y))\ud y\mtext{subject
      to}\int\limits_Y\varphi\ud y=0
  \end{equation*}
  has a unique minimizer in $\fs W^{1,2}_\per(Y;\Rrv n)$. If the energy
  density  $W$ additionally satisfies (W2) and is frame indifferent, then the
associated quadratic form $Q$ vanishes for skew symmetric
matrices. In this case $Q$ and its
  homogenization $Q_\hom$ are energy densities as they typically
  appear in linear
  elasticity.
\end{remark}

Theorem~\ref{thm:1} shows that the expansion property (W4) is stable
under homogenization for energy
  densities in   $\mathcal
  W(a,p)$. The next result states that also the non-degeneracy condition is
stable under homogenization:
\begin{lemma}\label{lem:2}
  Let ${W\in\mathcal W(a,p)}$, then $W\mc\in\mathcal W(a^\prime,p)$ for
  a positive constant  $a^\prime=a^\prime(a,n)$. Additionally, the map $W\mc:\Rrm n\rightarrow
  [0,\infty)$ is continuous and  quasiconvex.
\end{lemma}

In the language of $\Gamma$-convergence Theorem~\ref{thm:1} implies that
the $\Gamma$-limits of $(\mathcal I^{\varepsilon,h})$ corresponding to
linearization and homogenization commute. In order to state this in a
precise manner, let $\gamma$ be a closed subset of
$\partial\Omega$ with positive  {$n{-}1$-dimensional} Hausdorff-measure.
We denote the space of scaled displacements that satisfy the  Dirichlet
boundary condition on $\gamma$ by
\begin{equation*}
  \boundaryset:=\Big\{\,g\in\fs W^{1,2}(\Omega;\Rrv n)\,:\,g=0\text{ on }\gamma\,\Big\}.
\end{equation*}
For simplicity we assume that $\gamma$ is regular enough to guarantee
that  $\boundaryset\cap\fs W^{1,\infty}(\Omega;\Rrv n)$ is strongly
dense in $\boundaryset$. We consider the following functionals from $\fs
L^2(\Omega;\Rrv n)$ to $[0,+\infty]$:
\begin{align*}
  \energy I^{h,\varepsilon}(g)&:=\left\{\begin{aligned}
      &\frac{1}{h^2}\int\limits_\Omega W(x/\varepsilon,\id+h\grad g(x))\ud x &\qquad&\text{if }g\in\boundaryset\\
      &+\infty &&\text{else,}
    \end{aligned}\right.\\
  \energy I^{h}_\hom(g)&:=\left\{\begin{aligned}
      &\frac{1}{h^2}\int\limits_\Omega W\mc(\id+h\grad g(x))\ud x &\qquad&\text{if }g\in\boundaryset\\
      &+\infty &&\text{else,}
    \end{aligned}\right.\\
  \energy I^{\varepsilon}_\lin(g)&:=\left\{\begin{aligned}
      &\int\limits_\Omega Q(x/\varepsilon,\grad g(x))\ud x &\qquad&\text{if }g\in\boundaryset\\
      &+\infty &&\text{else,}
    \end{aligned}\right.\\
  \energy I^0(g)&:=\left\{\begin{aligned}
      &\int\limits_\Omega Q\super 1_\hom(\grad g(x))\ud x&\qquad&\text{if }g\in\boundaryset\\
      &+\infty &&\text{else.}
    \end{aligned}\right.
\end{align*}
Our second main result is the following:
\begin{theorem}\label{thm:4}
  Let $W\in\mathcal W(a,p)$ with $p\geq 2$. Then the following diagram commutes
  \begin{equation}\label{diagram}
    \begin{CD}
      \energy I^{\varepsilon,h}@>{(1)}>>\energy I^{\varepsilon}_\lin\\
      @V{(2)}VV  @VV{(3)}V\\
      \energy I^{h}_\hom@>>{(4)}>\energy I^{0}
    \end{CD}
  \end{equation}
  where $(1), (4)$ and $(2),(3)$ mean $\Gamma$-convergence with respect to strong convergence in $\fs L^2(\Omega;\Rrv n)$ as $h\rightarrow 0$ and $\varepsilon\rightarrow 0$, respectively.
\end{theorem}

In the diagram the $\Gamma$-limit (1), which corresponds to
linearization of a heterogeneous energy, was treated by G.~Dal~Maso,
M.~Negri and D.~Percivale in \cite{DalMaso2002}. In
Section~\ref{sec:lin} we give a slight variant of their argument,
which is adapted to assumption (W4).

\begin{remark}
On the level of the energy functional $\energy I^{\varepsilon,h}$,  it is natural to study
  also the $\Gamma$-convergence behavior as $\varepsilon$ and $h$ simultaneously
  converge to $0$. This corresponds to a ``diagonal limit'' in diagram \eqref{diagram}. Indeed, in \cite{NeukammPhD} the second author
  proves (based on two-scale convergence methods) that $\energy
  I^{\varepsilon,h}$ $\Gamma$-converges to $\energy I^0$ as
  $(\varepsilon,h)\rightarrow (0,0)$ simultaneously.
  Theorem~\ref{thm:4} is also related to recent works by A.~Braides and B.~Schmidt \cite{Braides2007,Schmidt2009} where the
passage from pair-interaction atomistic models to linear elasticity is
studied. In contrast to the setting considered in the present paper,
where $\varepsilon$ describes the length scale of the material
heterogeneity, in \cite{Braides2007,Schmidt2009} the small scale $\varepsilon$ originates from the
discrete nature of the atomistic model and measures the
interatomic distance. In \cite{Schmidt2009} the passage from-discrete-to-continuous is
obtained (in the regime $\varepsilon\ll h^2$) as a ``diagonal limit'', i.e. as a $\Gamma$-convergence result as $\varepsilon$ and $h$
simultaneously converge to $0$. It is observed that the derived $\Gamma$-limit coincides
with the model obtained by linearizing the continuum model derived
from the atomistic energy by applying the Cauchy-Born rule. The central
assumption is a discrete version of the non-degeneracy condition (W3).
In the atomistic setting this ensures that for sufficiently small displacements no
oscillations on the length scale of the lattice emerge.
\end{remark}

The $\Gamma$-convergence result is complemented by the
following equi-coercivity statement:
\begin{proposition}\label{prop:2}
  Suppose that ${W\in\mathcal W(a,p)}$ with $p\geq 2$.
  Then there exists a positive constant $c_1$ such that
  \begin{equation*}
    \min\Big\{\,\mathcal I^{\varepsilon,h}(g),\,\mathcal
    I^{\varepsilon}_\lin(g),\,\mathcal I^{h}_\hom(g),\,\mathcal I^0(g)\,\Big\}\geq c_1\Psi(g).
  \end{equation*}
  for all $\varepsilon,h>0$ and $g\in\fs L^2(\Omega;\Rrv n)$, where
  \begin{equation*}
    \Psi(g):=\left\{
      \begin{aligned}
        &\norm{g}^2_{\fs W^{1,2}(\Omega;\Rrv n)}&&\text{if
        }g\in\mathcal A_\gamma\\
        &+\infty&&\text{else.}
      \end{aligned}\right.
  \end{equation*}
\end{proposition}

\begin{remark}
  Obviously, the map $\Psi$ (and the restriction
  $\Psi\big\vert_{\boundaryset}$) is coercive and lower semicontinuous
  in $\fs L^2(\Omega;\Rrv n)$ (and coercive and lower
  semicontinuous with respect to weak convergence in
  $\fs W^{1,2}(\Omega;\Rrv n)$ respectively); thus, the previous
  proposition implies that the energies $\mathcal I^{\varepsilon,h}$,
  $\mathcal I^{\varepsilon}_\lin$, $\mathcal I^{h}_\hom$ and $\mathcal
  I^0$ are equi-coercive in the strong topology of $\fs
  L^2(\Omega;\Rrv n)$.
\end{remark}

\begin{remark}
  The local $p$-Lipschitz condition in (W1) can be dropped (see M.~Ba\'ia and I.~Fonseca \cite{Fonseca2007}). Nevertheless, it is a  natural property of quasiconvex or
  even rank-one convex functions: Let $p\in[1,\infty)$ and $f:\,\Rrm
  n\rightarrow [0,\infty)$ with   {$f( F)\leq C(1+\abs{ F}^p)$}. If
  $f$ is rank-one convex, then we have
  \begin{equation*}
    \abs{f( F)-f( G)}\leq C(1+\abs{ F}^{p-1}+\abs{ G}^{p-1})\abs{ F- G}.
  \end{equation*}
  We would like to remark that the $p$-growth and coercivity condition is too
  restrictive for a direct application to finite elasticity, because
  we expect the behavior $W(y, F)\rightarrow +\infty$ as $
  F\rightarrow 0$ for realistic materials. In this direction, in
  \cite{NeukammPhD} the second author considers periodic integrands
  $W$ that only satisfy (W2) -- (W4). It turns out
  that Theorem~\ref{thm:4} remains valid, when $\mathcal
  I^h_\hom$ is replaced by the lower $\Gamma$-limit of $(\energy
  I^{\varepsilon,h})_\varepsilon$ as $\varepsilon\rightarrow 0$.
\end{remark}

The plan of the paper is as follows. In
Section~\ref{sec:stab-non-degen} we prove Lemma~\ref{lem:2} and show
that the non-degeneracy condition (W3) is stable under homogenization.
In Section~\ref{sec:equi-coerc-based} we prove
Proposition~\ref{prop:2}. In particular, we see that the
non-degeneracy condition combined with the Dirichlet boundary
condition imposed on $\gamma$ yield equi-coercivity of the energies.
The proof relies on an approach by G.~Dal Maso et al. in
\cite{DalMaso2002} and combines the geometric rigidity
estimate (see \cite{Fries2002}) with an estimate on $\gamma$ that
allows to eliminate free rotations (see 
\cite{DalMaso2002} and Lemma~\ref{lem:5} below). Moreover, we prove a simple rigidity estimate for periodic
variations (see Lemma~\ref{lem:1}).

Section~\ref{sec:exp} is devoted to the proof of Theorem~\ref{thm:1}
and in Section~\ref{sec:lin} we present a linearization theorem that
slightly extends results in \cite{DalMaso2002} and is tailor-made for condition (W4). 
In Section~\ref{sec:hlc} we discuss the diagram
\eqref{diagram} and prove Theorem~\ref{thm:4}. Finally, in
Section~\ref{sec:non-comm-expans-3} we present two examples for which
homogenization and linearization do not commute.

\subsubsection*{Notation.}
We denote the space of real $n{\times}n$-matrices by $\Rrm n$ and the subset of rotations, i.e. $F\in\Rrm n$ with $F\transpose F=\id$
and $\det F=1$, by $\SO n$. For matrices $ F,  G\in\Rrm n$ we define the inner product and induced norm by
\begin{equation*}
  \iprod{ F}{ G}:=\trace F\transpose G\mtext{and}\abs{ F}:=\sqrt{\iprod{ F}{ F}}\qquad\text{respectively.}
\end{equation*}
Let $k\in\Nn$. A measurable function $u:\,\Rrv n\rightarrow\Rr$ is called 
$kY$-periodic, if it satisfies  $u(y+z)=u(y)$ for almost every $y\in\Rrv
n$ and all $z\in k\Zz^n$. We define the function spaces
\begin{align*}
  \fs L^p_\per(kY)&:=\left\{\,u\in\fs L^p_{\operatorname{loc}}(\Rrv n)\,:\,u\text{ is
      $kY$-periodic}\,\right\},\\
  \fs W^{1,p}_\per(kY)&:=\left\{\,u\in\fs L^p_\per(kY)\,:\,u\in\fs
    W^{1,p}_{\operatorname{loc}}(\Rrv n)\,\right\}
\end{align*}
and likewise  $\fs L^p_\per(kY;\Rrv n)$ and $\fs W^{1,p}_\per(kY;\Rrv
n)$.

We denote the set of linear maps $\tensor L$ from $\Rrm n$ to $\Rrm n$
that satisfy  $\dprod{\tensor LA}{B}=\dprod{\tensor LB}{A}$ for all
$A,B\in\Rrm n$ by $\tsym(n)$, i.e. $\tensor L\in\tsym(n)$ is a
symmetric fourth order tensor. We associate to each quadratic
integrand $Q\in\mathcal Q$ a
unique map $\tensor L\in\fs L^\infty_\per(Y;\tsym(n))$ according to
\begin{equation*}
  \iprod{\tensor
    L(y)F}{G}=\frac{Q(y,F+G)-Q(y,F)-Q(y,G)}{2}\qquad\text{for all
  }F,G\in\Rrm n.
\end{equation*}

\section{Stability of the non-degeneracy condition. Proof of
  Lemma~\ref{lem:2}.}\label{sec:stab-non-degen}
In this section we prove Lemma~\ref{lem:2}. Let $W\in\mathcal W(a,p)$.
It is well-known that the homogenized integrand  $W\mc:\Rrm n\rightarrow[0,\infty)$ is a
continuous and quasiconvex map that satisfies the
$p$-growth and $p$-coercivity condition (W1) (see e.g. \cite{Mueller1987}, \cite{Braides1998} and \cite{Fonseca2007}). Condition (W2) is trivially fulfilled
and Theorem \ref{thm:1} implies that
$W\mc$ has a Taylor expansion at $\id$, and therefore satisfies condition (W4).
Thus, it remains to prove that the non-degeneracy condition (W3)
is stable under homogenization. This is an immediate consequence of
the following observation:

\begin{lemma}\label{lem:4}
  Let $W:\,\Rrv n{\times}\Rrm n\rightarrow [0,+\infty]$ be a
  Carath\'eodory function $Y$-periodic in its first variable and
  suppose that $W$ satisfies the non-degeneracy condition (W3). Then
  for all $F\in\Rrm n$ we have
  \begin{equation*}
    W\mc(F)\geq \frac{1}{a^\prime}\dist^2(F,\SO n)
  \end{equation*}
  where the positive constant $a^\prime$ only depends on the dimension
  $n$ and the constant from condition (W3).
\end{lemma}
\begin{proof}
  By definition we can find for each $k\in\Nn$ a
  number $m_k\in\Nn$ and a map $\psi_k\in\fs W^{1,p}_\per(m_kY;\Rrv n)$
  such that
  \begin{equation*}
    W\mc(F)+\frac{1}{k}\geq
    \frac{1}{m_k^n}\int\limits_{m_kY}W(y,F+\grad\psi_k(y))\ud y.
  \end{equation*}
  We apply the non-degeneracy condition (W3) to the right hand side and
  obtain
  \begin{equation*}
    W\mc(F)+\frac{1}{k}\geq
    \frac{1}{a}\frac{1}{m_k^n}\int\limits_{m_kY}\dist^2(F+\grad\psi_k(y),\SO n)\ud y.
  \end{equation*}
  By replacing the map $F\mapsto\dist^2(F,\SO n)$ by its
  quasiconvexification  $\mathrm{Q}\!\dist^2(\cdot,\SO n)$ we get a
  lower estimate:
  \begin{equation*}
    W\mc(F)+\frac{1}{k}\geq
    \frac{1}{a}\left(\frac{1}{m_k^n}\int\limits_{m_kY}\mathrm Q\!\dist^2(F+\grad\psi_k(y),\SO n)\ud y\right).
  \end{equation*}
  Because $m_kY$ is a quadratic domain and $\psi_k$ is
  $m_kY$-periodic, we see that the integral in the braces is bounded
  from below by $\mathrm Q\!\dist^2(F,\SO n)$ due to quasiconvexity. Thus,
  \begin{equation*}
    W\mc(F)+\frac{1}{k}\geq \frac{1}{a}\mathrm{Q}\!\dist^2(F,\SO n).
  \end{equation*}
  K. Zhang proved in \cite{Zhang1997} that the quasiconvexification
  $\mathrm{Q}\!\dist^2(\cdot,\SO n)$ can be bounded from below by
  $c_n\dist^2(\cdot,\SO n)$ where $c_n$ is a positive constant. Hence,
  we arrive at
  \begin{equation*}
    W\mc(F)+\frac{1}{k}\geq \frac{1}{a^\prime}\,\dist^2(F,\SO n)
  \end{equation*}
  where $a^\prime:=\frac{c_n}{a}>0$. Passing to the limit
  $k\rightarrow\infty$ completes the proof.
 \end{proof}

\section{Equi-coercivity based on geometric rigidity}\label{sec:equi-coerc-based}
In this section we prove Proposition~\ref{prop:2}. First, we would
like to remark that  Proposition~\ref{prop:2} is an equi-coercivity
statement. For instance the proposition implies that whenever a family of scaled displacements
$(g_{\varepsilon,h})_{\varepsilon,h}\subset\fs L^2(\Omega;\Rrv n)$ has equi-bounded energy, i.e. 
\begin{equation*}
  \limsup\limits_{(\varepsilon,h)\rightarrow (0,0)}\energy
  I^{\varepsilon,h}(g_{\varepsilon,h})<\infty,
\end{equation*}
then it is  relatively compact in $\fs L^2(\Omega;\Rrv n)$ and we can extract
a subsequence that strongly converges to  a map $g\in\boundaryset$ in
$\fs L^2(\Omega;\Rrv n)$. If additionally each element of the sequence
$(g_{\varepsilon,h})$ has finite energy, then the construction of $\Psi$ reveals that the relative compactness
also holds with respect to the weak topology in $\fs
W^{1,2}(\Omega;\Rrv n)$. In some sense this observation a priori
justifies the presentation of the scaled energy in terms of the scaled displacement.

A first step towards the proof of Proposition~\ref{prop:2} is to show that the
norm of a displacement gradient can be controlled by the associated
energy. Because of the non-degeneracy condition (W3), we want to
establish an estimate in the form
\begin{equation}\label{eq:18}
\forall g\in\boundaryset\,:\,\int\limits_\Omega\abs{\grad g}^2\ud x\leq
C\frac{1}{h^2}\int\limits_\Omega\dist^2(\id+h\grad g(x),\SO n)\ud x
\end{equation}
for a constant $C$ that is independent of  $h$.
A key ingredient in the proof of this estimate is the geometric
rigidity estimate by G.~Friesecke, R.D.~James and the first author:
\begin{theorem}[Geometric rigidity estimate \cite{Fries2002}]\label{thm:rigidity}
Let $U$ be a bounded Lipschitz domain in $\Rrv n$, $n\geq 2$. There exists a  constant $C(U)$ with the following property: For each $v\in\fs W^{1,2}(U;\Rrv n)$ there is an associated rotation $ R\in\SO n$ such that
 \begin{equation*}
  \int\limits_U\abs{\grad v(x)- R}^2\ud x\leq C(U)\int\limits_U\dist^2(\grad v(x),\SO n)\ud x.
 \end{equation*}
Moreover, the constant $C(U)$ is invariant under uniform scaling of $U$.
\end{theorem}
In virtue of this result we can assign to each $g\in\boundaryset$ and
positive parameter $h$  a single rotation ${R\in\SO n}$
such that
\begin{equation*}
  \norm{\frac{\id+h\grad g(x)-R}{h}}^2_{\fs L^2(\Omega;\Rrm n)}\leq
  C(\Omega)\frac{1}{h^2}\int\limits_\Omega\dist^2(\id+h\grad g(x),\SO n)\ud x.
\end{equation*}
Let us assume for a moment that $R=\id$. Then the previous estimate and
the non-degeneracy of $W$ directly imply \eqref{eq:18}. In general a similar
observation is valid: By taking the Dirichlet boundary condition imposed on
$\gamma$ into account we can eliminate the free rotation. This has been
shown by G.~Dal Maso et al. in \cite{DalMaso2002}. In particular, they
proved the following result:
\begin{lemma}[see Lemma 3.3 in \cite{DalMaso2002}]\label{lem:5}
Let $\Omega$ be an open and bounded Lipschitz domain in $\Rrv n$ and
$\gamma$ a measurable subset of $\partial\Omega$ with positive
$(n-1)$-dimensional Hausdorff measure. For $F\in\Rrm n$ define
\begin{equation*}
  \abs{F}_\gamma^2:=\min\limits_{b\in\Rrv
    n}\int\limits_{\gamma}\abs{Fx-b}^2\,\mathrm{d}\mathcal H^{n-1}(x).
\end{equation*}
Then there exists a positive constant $C$ such that
\begin{equation*}
  \abs{F}^2\leq C \abs{F}_\gamma^2
\end{equation*}
for all matrices $F$ that belong to the union of the cone generated by $\id-\SO
n$ and the set of skew symmetric matrices in $\Rrm n$.  
\end{lemma}
In the lemma above $\mathcal H^{n-1}$ denotes the $(n-1)$-dimensional
Hausdorff measure. With this result at hand we are in position to prove the proposition:
\begin{proof}[Proof of Proposition~\ref{prop:2}]
 It follows from the proof of Proposition 3.4 in  \cite{DalMaso2002} that
\begin{equation}\label{eq:19}
  \norm{g}^2_{\fs W^{1,2}(\Omega;\Rrv n)}\leq
  c^\prime\frac{1}{h^2}\int\limits_\Omega\dist^2(\id+h\grad g(x),\SO n)\ud
  x
\end{equation}
for all $g\in\boundaryset$ and all $h\in(0,1)$.
Here and below $c^\prime$ denotes a positive constant that may vary from line to line, but can be chosen only depending on $W$ and the geometry of
  $\Omega$ and $\gamma$. For the sake of completeness we
  briefly sketch the proof: In
  virtue of Theorem~\ref{thm:rigidity} we can assign to each
  $g\in\boundaryset$ and $h\in(0,1)$ a rotation $R\in\SO n$ such that
  \begin{equation*}
    \int\limits_\Omega \abs{\id+h\grad g(x)-R}^2\leq
    c^\prime\dist^2(\id+h\grad g(x),\SO n)\ud x.
  \end{equation*}
  Since  $h\grad g=(\id+h\grad g-R)-(\id-R)$, we have
  \begin{equation*}
    \norm{g}^2_{\fs W^{1,2}(\Omega;\Rrv n)}\leq
    c^\prime\frac{1}{h^2}\int\limits_\Omega\dist^2(\id+h\grad g(x),\SO
    n)\ud x+2\norm{\frac{R-\id}{h}}^2_{\fs L^2(\Omega;\Rrv n)}.
  \end{equation*}
  Set $u(x):=(\id - R)x+hg(x)-u_\Omega$ where $u_\Omega\in\Rrv n$ is
  chosen in such a way that $u$ is mean value free. Because $g(x)=0$
  on $\gamma$, we have $u(x)=(\id-R)x-u_\Omega$ on $\gamma$ and
  Lemma~\ref{lem:5} implies that
  \begin{equation*}
    \abs{\id-R}^2\leq C\abs{\id-R}^2_\gamma\leq C
    \int\limits_\gamma\abs{u(x)}^2\,\mathrm d\mathcal H^{n-1}(x)\leq c^{\prime}\int\limits_\Omega\abs{\grad u(x)}^2\ud x,
  \end{equation*}
  where we used the continuity of the trace operator and
  Poincar\'e-Wirtinger inequality.  
  Because of the identity $\grad u=\id+h\grad g-R$, the right hand side is controlled
  by $\int_\Omega\dist^2(\id+h\grad g(x),\SO n)\ud x$ and
  \eqref{eq:19} follows.
  
  By appealing to the non-degeneracy condition (W3) and Lemma~\ref{lem:4}, we
  immediately see that
  \begin{equation*}
    \min\{\,\energy I^{\varepsilon,h}(g),\,\energy I^h_\hom(g)\,\}\geq
    c^\prime\frac{1}{h^2}\int\limits_\Omega\dist^2(\id+h\grad g(x),\SO n)\ud x
  \end{equation*}
  for all $g\in\boundaryset$ and all $\varepsilon,h>0.$   Since  the energies are infinite
  whenever $g\not\in\boundaryset$, the previous estimate and
 \eqref{eq:19} imply that
  \begin{equation*}
    \min\{\,\mathcal I^{\varepsilon,h}(g),\,\mathcal I^{h}_\hom(g)\}\geq c^\prime\Psi(g).
  \end{equation*}
  for all $\varepsilon,h>0$ and $g\in\fs L^2(\Omega;\Rrv n)$.

  Next, we consider the energies $(\energy I^\varepsilon_\lin)$ and
  $\energy I^0$. It is not difficult to show that (W3) and (W4) imply that
  \begin{equation*}
    \min\left\{\,Q(y,F),\,Q\super 1_\hom(F)\,\right\}\geq c^\prime\abs{\sym F}^2\qquad\text{for all $F\in\Rrm n$
      and a.e. $y\in\Rrv n$.}
  \end{equation*}
  Thus, it is sufficient to prove that
  \begin{equation}\label{eq:20}
    \norm{g}^2_{\fs W^{1,2}(\Omega;\Rrv n)}\leq c^\prime
    \int\limits_\Omega\abs{\sym\grad g(x)}^2\ud x\qquad\text{for all }g\in\boundaryset.
  \end{equation}
  This can be seen as follows: Let $g\in\boundaryset$. By Korn's
  inequality there exists a skew symmetric matrix
  $K\in\Rrm n$  such that
  \begin{equation*}
    \norm{\grad g-K}^2_{\fs L^2(\Omega;\Rrm n)}\leq
    c^\prime\int\limits_\Omega\abs{\sym\grad g(x)}^2\ud x.
  \end{equation*}
  We set $u(x):=g(x)-Kx-u_\Omega$ where $u_\Omega\in\Rrv n$ is chosen in
  such a way that $u$ has vanishing mean value. As before we can apply
  Lemma~\ref{lem:1} and find
  \begin{equation*}
    \abs{K}^2\leq C\abs{K}_\gamma^2\leq
    c^\prime\int\limits_{\Omega}\abs{\grad u}^2\ud x.
  \end{equation*}
  Since $\grad u=\grad g-K$, we find that
  \begin{equation*}
    \norm{g}_{\fs W^{1,2}(\Omega;\Rrv n)}\leq c^\prime\norm{\grad
      g}_{\fs L^2(\Omega;\Rrm n)}\leq c^\prime\left(\,\norm{\grad u}_{\fs
      L^2(\Omega;\Rrm n)}+\abs{K}\,\right).
  \end{equation*}
  Now \eqref{eq:20} follows because the right hand side is controlled
  by $\norm{\sym\grad g}_{\fs L^2(\Omega;\Rrm n)}$.
 \end{proof}

\begin{remark}\label{sec:see-lemma-3.3}
  In the proof we did not use the property that $W$ satisfies the growth and
  coercivity condition (W1). Moreover, we could deduce from
  Theorem~\ref{thm:4} and general properties of $\Gamma$-convergence,
  that $(\mathcal I^h_\hom)$ and $(\mathcal I^\varepsilon_\lin)$ and
  $\mathcal I^0$ are equi-coercive whenever $(\mathcal
  I^{\varepsilon,h})$ is equi-coercive. 
\end{remark}

Another useful observation is the following:
\begin{lemma}\label{lem:1}
 There exists a positive constant $c$ such that
\begin{align*}
 \frac{1}{h^2}\int\limits_{k Y}\dist^2(\id+h\grad \psi(x),\SO n)\ud x\geq c\int\limits_{kY}\abs{\grad\psi}^2\ud y
\end{align*}
for all $h>0$, $k\in\Zz$ and maps $\psi\in\fs W^{1,2}_\per(kY;\Rrv n)$.
\end{lemma}
\begin{proof}[Proof of Lemma \ref{lem:1}]
Set $v(x){:=}x+h\psi(x)$. By Theorem \ref{thm:rigidity} there exists a rotation $R\in\SO n$ satisfying
 \begin{equation*}
	\int\limits_{kY}\abs{\frac{ R-\id}{h}-\grad\psi(y)}^2\ud y\leq C\frac{1}{h^2}\int\limits_{kY}\dist^2(\id+h\grad\psi(y),\SO n)\ud y.
 \end{equation*}
The constant $C$ is independent of $h$, $k$ and  $\psi$, because the
domain $kY$ is obtained by uniformly scaling the cell $Y$. We expand
the left hand side and see that
\begin{equation*}
   \int\limits_{kY}\abs{\frac{R{-}\id}{h}+\grad\psi(y)}^2\ud y= \int\limits_{kY}\abs{\frac{R-\id}{h}}^2+2\iprod{\frac{R-\id}{h}}{\grad\psi(y)}+\abs{\grad\psi(y)}^2\ud y.
\end{equation*}
Because gradients of functions in $\fs W^{1,2}_\per(kY;\Rrv n)$ are
orthogonal to constant matrices (with respect to the standard inner
product in $\fs L^2(kY;\Rrm n)$), we deduce that the integral over the coupling term in the middle vanishes and we immediately obtain
 \begin{equation*}
	\int\limits_{kY}\abs{\grad\psi(y)}^2\ud y\leq C\frac{1}{h^2}\int\limits_{kY}\dist^2(\id+h\grad\psi(y),\SO n)\ud y.
 \end{equation*}
 \end{proof}

\section{Expansion of the multi-cell homogenization formula. Proof of
  Theorem \ref{thm:1}.}\label{sec:exp}
For convenience we define for $x\in\Rrv n$ and $G\in\Rrm
n\setminus\{0\}$ the remainders 
\begin{equation}\label{eq:3}
  \restterm(x,G):=\frac{W(x,\id+G)-Q(x,G)}{\abs{G}^2}
\end{equation}
and
\begin{equation*}
  \restterm_\hom(G):=\frac{W\mc(\id+G)-Q\super 1_\hom(G)}{\abs{G}^2}.
\end{equation*}
In order to prove Theorem~\ref{thm:1} it is sufficient to show the
following: For any sequence of matrices $(G_k)$ in $\Rrm n\setminus\{0\}$ with
$\abs{G_k}\rightarrow 0$ there holds
\begin{equation}\label{eq:12}
  \limsup\limits_{k\rightarrow\infty}\restterm_\hom(G_k)=0.
\end{equation}
Because the normalized sequence $(\abs{G_k}^{-1}G_k)$ is relatively compact, it
is sufficient to consider sequences that additionally satisfy
\begin{equation}\label{eq:1}
  H_k:=\frac{G_k}{\abs{G_k}}\rightarrow G\qquad\text{in }\Rrm n\text{
    as $k\rightarrow \infty$}.
\end{equation}
In the sequel we separately prove that
\begin{align}
  \label{eq:13}&\limsup\limits_{k\rightarrow
    0}\frac{1}{\abs{G_k}^2}W\mc(\id+G_k)\leq Q\super 1_\hom(G),\\
  \label{eq:14}&\liminf\limits_{k\rightarrow
    0}\frac{1}{\abs{G_k}^2}W\mc(\id+G_k)\geq Q\super 1_\hom(G).
\end{align}
Clearly, the validity of both estimates is equivalent to
\eqref{eq:12}.

\step 1 We prove the \textbf{upper bound} estimate \eqref{eq:13}.
Because $Q$ is a Carath\'eodory function quadratic in its second
variable, the functional
\begin{equation*}
  W^{1,2}_\per(Y;\Rrv n)\ni\varphi\mapsto
  \int_YQ(y,G+\grad\varphi(y))\ud y
\end{equation*}
is lower semicontinuous with respect to weak convergence and continuous
with respect to strong convergence. Furthermore, the strict convexity
of $Q(x,\cdot)$ on the subspace of symmetric matrices is sufficient to
guarantee that the functional admits a minimizer in $\fs
W^{1,2}_\per(Y;\Rrv n)$ (see Remark~\ref{rem:1}). The
inclusion  $\fs C_\per^\infty(Y;\Rrv n)\subset\fs W^{1,2}_\per(Y;\Rrv
n)$ is dense; thus, by the strong continuity of the functional we find for
every $\eta>0$ a map $\psi\in\fs C_\per^\infty(Y;\Rrv n)$
such that
\begin{equation*}
  Q\super 1_\hom(G)=\min\limits_{\varphi\in\fs W^{1,2}_\per(Y;\Rrv
    n)}\int\limits_YQ(y,G+\grad\varphi(y))\ud y\geq
  \int\limits_YQ(y,G+\grad\psi(y))\ud y -\eta.
\end{equation*}
Based upon this choice we derive an upper bound for the left hand
side in \eqref{eq:2}: By construction the multi-cell homogenization
 $W\mc(\cdot)$ is bounded from above by the one-cell homogenization $W\super
1_\hom(\cdot)$; thus, we obtain
\begin{equation*}
  W\mc(\id+ G_k)\leq W\super 1_\hom(\id +G_k)\stackrel{(\star)}{\leq} \int\limits_Y W\left(y,\id +
  G_k+\abs{G_k}\grad\psi(y)\right)\ud y.
\end{equation*}
Inequality $(\star)$ follows from directly follows from the definition
of $W\super 1_\hom$. We expand the integrand on
the right hand side and deduce that for almost every $y\in Y$ we have
\begin{equation*}
  \frac{1}{\abs{G_k}^2}W(y,\id+G_k+\abs{G_k}\grad\psi(y))\leq Q(y,H_k+\grad\psi(y))+\restterm(y,G_k+\abs{G_k}\grad\psi(y))
\end{equation*}
where the remainder $\restterm$ is defined according to \eqref{eq:3}.
Because $(G_k)$ vanishes and $\psi\in\fs W^{1,\infty}(Y;\Rrv n)$, we have
\begin{equation*}
  G_k+\abs{G_k}\grad\psi(y)\rightarrow 0\qquad\text{uniformly};
\end{equation*}
thus, condition (W4) implies that
\begin{equation*}
  \limsup\limits_{k\rightarrow\infty}\int\limits_Y\restterm(y,G_k+\abs{G_k}\grad\psi(y))\ud y=0.
\end{equation*}
Consequently, the previous estimates, the convergence $H_k\rightarrow G$ and the continuity
of $Q$ lead to
\begin{equation*}
  \limsup\limits_{k\rightarrow\infty}\frac{1}{\abs{G_k}^2}W\mc(\id+G_k)\leq
  \int\limits_YQ(y,G+\grad\psi(y))\ud y\leq Q\super 1_\hom(G)+\eta.
\end{equation*}
Because $\eta>0$ can be chosen arbitrarily small, \eqref{eq:2} follows.

\step 2 We prove the \textbf{lower bound} estimate \eqref{eq:14}. We
only have to consider the case where
\begin{equation*}
  \liminf\limits_{k\rightarrow\infty}\frac{1}{\abs{G_k}^2}W\mc(\id+G_k)
\end{equation*}
is finite. We pass to a subsequence, that we do not relabel, such that
\begin{equation*}
  \liminf\limits_{k\rightarrow 0}\frac{1}{\abs{G_k}^2}W\mc(\id+G_k)=\limsup\limits_{k\rightarrow 0}\frac{1}{\abs{G_k}^2}W\mc(\id+G_k).
\end{equation*}
By definition, for all $k\in\Nn$ there exist a number
$m_k\in\Nn$ and a map
\begin{equation*}
  \psi_k\in\fs W^{1,2}_\per(m_kY;\Rrv
  n)\mtext{with}\int\limits_{m_kY}\psi_k(y)\ud y=0
\end{equation*}
such that
\begin{equation}\label{eq:6}
  W\mc(\id+G_k)+\frac{\abs{G_k}^2}{k}\geq \frac{1}{m_k^n}\int\limits_{m_kY}W(y,\id
  +G_k+\abs{G_k}\grad \psi_k(y))\ud y.
\end{equation}
This suggests to establish the lower bound estimate by applying the expansion
in condition (W4) to the integral on the right hand side. Clearly, if $\abs{\grad\psi_k(y)}$ was bounded by a
constant independent of $y$ and $k$, condition (W4) would immediately
imply \eqref{eq:14}. However, the sequence $(\grad\psi_k)$ is only bounded in the
following sense:
\begin{equation}\label{eq:4}
  C:=\limsup\limits_{k\rightarrow\infty}\frac{1}{m_k^n}\int\limits_{m_kY}\abs{\grad\psi_k}^2\ud
  y<\infty,
\end{equation}
as can be seen by appealing to the non-degeneracy condition (W3) and Lemma~\ref{lem:1}. Therefore, we distinguish points $y\in m_kY$ where
$\abs{\grad\psi_k(y)}$ is sufficiently small from those where
$\abs{\grad\psi_k(y)}$ is too large for an expansion. More precisely,
we define the set
\begin{equation}\label{eq:26}
  Y_k:=\left\{\,y\in m_kY\,:\,\abs{\grad\psi_k(y)}\leq \abs{G_k}^{-1/2}\,\right\}
\end{equation}
and denote the associated indicator function by $\chi_k$. The proof of
the lower bound estimate is divided in two steps. First, we show that
\begin{equation}\label{eq:5}
  \liminf\limits_{k\rightarrow\infty} \frac{1}{\abs{G_k}^2}W\mc(\id+G_k)\geq
  \liminf\limits_{k\rightarrow
    \infty}\frac{1}{m_k^n}\int\limits_{m_kY}Q\left(y,\chi_k(H_k+\grad\psi_k)\right)\ud y
\end{equation}
and in a second step, we prove that
\begin{equation}\label{eq:11}
  \liminf\limits_{k\rightarrow\infty}\frac{1}{m_k^n}\int\limits_{m_kY}
    Q(y,\chi_k(H_k+\grad\psi_k))\ud y-Q\super 1_\hom(G)\geq 0.
\end{equation}
It is obvious that the combination of both estimates justifies \eqref{eq:14}.

\step{3} (Proof of  \eqref{eq:5}). By construction we have
\begin{equation*}
  \esssup\limits_{y\in
    Y_k}\Big\vert\,G_k+\abs{G_k}\grad\psi_k(y)\,\Big\vert\rightarrow 0
\end{equation*}
and in view of condition (W4) we see that
\begin{equation}\label{eq:7}
  \limsup\limits_{k\rightarrow
    0}\frac{1}{m_k^n}\int\limits_{m_kY}\chi_k\abs{\frac{W(y,\id+G_k+\abs{G_k}\grad\psi_k)}{\abs{G_k}^2}-Q(y,H_k+\grad\psi_k)}\ud y=0.
\end{equation}
Now the non-negativity of $W$ implies that
$W(y,F)\geq\chi_k(y)W(y,F)$ for almost every $y\in m_kY$ and all
$F\in\Rrm n$; thus, estimate \eqref{eq:6} and \eqref{eq:7} immediately imply
that
\begin{equation*}
  \liminf\limits_{k\rightarrow\infty}
  \frac{1}{\abs{G_k}^2}W\mc(\id+G_k)\geq
  \liminf\limits_{k\rightarrow\infty}\frac{1}{m_k^n}\int\limits_{m_kY}\chi_kQ(y,H_k+\grad\psi_k)\ud y.
\end{equation*}
Because $\chi_k$ takes only values in $\{0,1\}$, we see that 
\begin{equation*}
  \chi_k(y)Q\Big(y,H_k+\grad\psi_k(y)\Big)=Q\Big(y,\chi_k(y)(H_k+\grad\psi_k(y))\Big)
\end{equation*}
for all $y\in m_kY$ and \eqref{eq:5} follows.

\step 4 (Proof of \eqref{eq:11}). This is the heart of the matter. Obviously, if the indicator function $\chi_k$ was equal to $1$, then
the integral in \eqref{eq:11} would be bounded from below by $Q\super
1_\hom(H_k)$ and the estimate would follow from the continuity of
$Q\super 1_\hom$. The general case would follow if we knew that
\begin{equation*}
  \frac{1}{m_k^n}\int_{m_kY}(1-\chi_k)\abs{H_k+\grad\psi_k}^2\ud
  x\rightarrow 0.
\end{equation*}
Because a priori $\abs{\grad\psi_k}^2$  could concentrate on the set
where $\chi_k=0$, this is not obvious at all. Since we aim for a lower
bound, we expand $Q(y,H_k+\grad\psi_k)$ around the minimizer
$G+\grad\psi_G$ (see below). Then the most dangerous quadratic term
has a sign.

Thus, let $\psi_G\in\fs W^{1,2}_\per(Y;\Rrv n)$ satisfy
\begin{equation*}
  Q\super 1_\hom(G)=\int\limits_YQ(y,G+\grad\psi_G(y))\ud y.
\end{equation*}
We extend $\psi_G$ by periodicity to $\Rrv n$. Since  $Q$ is
$Y$-periodic and convex, it is not difficult to
check that $\psi_G$ is also a minimizer of the functional
\begin{equation}\label{eq:9}
  \fs W^{1,2}_\per(m_kY;\Rrv n)\ni \psi\mapsto
  \int\limits_{m_kY}Q(y,G+\grad\psi(y))\ud y.
\end{equation}
 Because $Q(y,\cdot)$ is a quadratic form, the inequality
\begin{equation*}
  Q(y,A)-Q(y,B)\geq 2\dprod{\tensor L(y)(A-B)}{B}
\end{equation*}
is valid for all $A,B\in\Rrm n$ and almost every $y$. We apply this
inequality with
\begin{equation*}
  A=\chi_k(y)(H_k+\grad\psi_k(y))\mtext{ and }B=G+\grad\psi_G(y).
\end{equation*}
Now integration over $m_kY$ leads to
\begin{multline*}
 \frac{1}{2\,m_k^n}\int\limits_{m_kY}Q(y,\chi_k(H_k+\grad\psi_k))-Q(y,G+\grad\psi_G)\ud
  y\\
  \geq  \frac{1}{m_k^n}\int\limits_{m_kY}\dprod{\tensor
    L(y)\Big[(\chi_k(H_k+\grad\psi_k)-(G+\grad\psi_G)\Big]}{G+\grad\psi_G}\ud y.
\end{multline*}
The right hand side can be rewritten as
\begin{align*}
  &\frac{1}{m_k^n}\int\limits_{m_kY}\dprod{\tensor
    L(y)(H_k{-}G)}{G{+}\grad\psi_G}\ud
  y\\
  +&\frac{1}{m_k^n}\int\limits_{m_kY}\dprod{\tensor
    L(y)(\grad\psi_k{-}\grad\psi_G)}{G{+}\grad\psi_G}\ud
  y\\
  -&\frac{1}{m_k^n}\int\limits_{m_kY}\dprod{\tensor
    L(y)(1-\chi_k)(H_k+\grad\psi_k)}{G+\grad\psi_G}\ud
  y\\
  =:&I_k\super 1+I_k\super2+I_k\super3.
\end{align*}
In the following we prove that all three integrals vanish as
$k\rightarrow\infty$. We start with the first integral $I\super1_k$. Due to the
$Y$-periodicity of $\tensor L$ and $\psi_G$, we have
\begin{equation*}
  I\super1_k=\int\limits_Y\dprod{\tensor
    L(y)(H_k{-}G)}{G{+}\grad\psi_G}\ud
  y
\end{equation*}
and \eqref{eq:1} implies that $I\super1_k\rightarrow 0$ as
$k\rightarrow\infty$. We consider the second integral $I\super 2_k$.
It is easy to check that $I\super2_k$ is exactly the Euler Lagrange
equation of the quadratic minimization problem
associated to the functional \eqref{eq:9}. Because $\psi_G$ is a
minimizer and the map $y\mapsto \psi_k-\psi_G$ an admissible test
function, we have $I\super2_k=0$ for all $k\in\Nn$.

Finally, we consider the third integral $I\super3_k$. By applying the
Cauchy-Schwarz- and H\"older-inequality, we find that
\begin{equation}\label{eq:10}
  \abs{I\super3_k}^2\leq c^\prime
  \left(\frac{1}{m_k^n}\int\limits_{m_kY}\abs{H_k+\grad\psi_k}^2\ud y\right)\left(\frac{1}{m_k^n}\int\limits_{m_kY}\abs{(1-\chi_k)(G+\grad\psi_G)}^2\ud y\right)
\end{equation}
where $c^\prime$ is a positive constant independent of $k$. In virtue
of \eqref{eq:4}, it is sufficient to prove that the second integral
vanishes as $k\rightarrow \infty$. By construction the multi-cell $m_kY$ is the
disjoint union of the $m_k^n$ translated cells $Y+\xi$ with $\xi\in
Z_k:=m_kY\cap\Zz^n$. Consequently, we can rewrite the second integral
in \eqref{eq:10} according to 
\begin{multline*}
\frac{1}{m_k^n}\int\limits_{m_kY}\abs{(1-\chi_k)(G+\grad\psi_G)}^2\ud
    y\\=\frac{1}{m_k^n}\sum_{\xi\in Z_k}\int\limits_Y\Big(1-\chi_k(y+\xi)\Big)\Big\vert
    G+\grad\psi_G(y+\xi)\Big\vert^2\ud y
\end{multline*}
where we used the fact that the map $y\mapsto (1-\chi_k)$ only takes values in $\{0,1\}$.
Because the map $\psi_G$ is $Y$-periodic, the right hand side
simplifies to
\begin{equation*}
\int\limits_Y\bar\chi_k(y)\abs{G+\grad\psi_G(y)}^2\ud y\mtext{with}
  \bar\chi_k(y):=\frac{1}{m_k^n}\sum_{\xi\in Z_k}\Big(\,1-\chi_k(y+\xi)\,\Big).
\end{equation*}
Recall that $\chi_k$ denotes the indicator function of the set $Y_k$
defined in \eqref{eq:26}. By definition there holds
$1\leq \abs{G_k}\abs{\grad\psi_k(y)}^2$ for a.e. $y\in m_kY\setminus Y_k$. Thus, we estimate
\begin{equation*}
  \int\limits_Y\abs{\bar\chi_k(y)}\ud y\leq
  \frac{1}{m_k^n}\int\limits_{m_kY}(1-\chi_k(y))\ud y\leq
  \frac{1}{m_k^n}\int\limits_{m_kY}\abs{G_k}\abs{\grad \psi_k(y)}^2\ud y\leq C\abs{G_k}
\end{equation*}
where $C$ denotes the constant from \eqref{eq:4}. Consequently, the
sequence $\bar\chi_k$  strongly converges to $0$ in $\fs L^1(Y)$.
 We claim that
\begin{equation}\label{eq:8}
  \bar\chi_k\abs{G+\grad\psi_G}^2\rightarrow 0\qquad\text{strongly in }L^1(Y).
\end{equation}
Since  $\bar\chi_k\rightarrow 0$ in $\fs L^1(Y)$, the left hand side
converges to $0$ in measure. On the other hand, by construction we
have $0\leq\bar\chi_k\leq 1$, and therefore the left hand side is
dominated by the map $\abs{G+\grad\psi_G(\cdot)}^2$ which belongs to
$\fs L^1(Y)$. Thus, by dominated convergence  assertion \eqref{eq:8}
follows and $I\super 3_k$ vanishes as $k\rightarrow 0$. So far we have shown that
\begin{equation*}
  \frac{1}{m_k^n}\int\limits_{m_kY}Q(y,\chi_k(H_k+\grad\psi_k))-Q(y,G+\grad\psi_G(y))\ud
  y\rightarrow 0
\end{equation*}
as $k\rightarrow \infty$. In virtue of the $Y$-periodicity and
convexity of $Q$, we have
\begin{equation*}
  Q\super 1_\hom(G)=\frac{1}{m_k^n}\int\limits_{m_kY}Q(y,G+\grad\psi_G(y))\ud
  y
\end{equation*}
and \eqref{eq:11} follows.\qed

\section{Linearization.}\label{sec:lin}
G.~Dal Maso et al. proved in \cite{DalMaso2002} that linearized
elasticity can be obtained as a $\Gamma$-limit from nonlinear,
three-dimensional elasticity. The following theorem is a variant of their result adapted to assumption (W4).
\begin{theorem}\label{thm:2}
 Let $W:\,\Omega{\times}\Rrm n\rightarrow[0,\infty)$ be a
 Carath\'eodory function and suppose that
 \begin{equation}\label{eq:50}
   \limsup\limits_{G\rightarrow 0,\atop G\neq 0}\frac{\abs{W(x,\id+G)-Q(x,G)}}{\abs{G}^2}=0,
 \end{equation}
 where $Q:\,\Omega{\times}\Rrm n\rightarrow[0,\infty)$ is a
 Carath\'eodory function quadratic in its second variable and
bounded in the sense that 
\begin{equation*}
 \forall G\in\Rrm n\,:\,\esssup_{x\in\Omega}Q(x,G)\leq c_1\abs{G}^2 
\end{equation*}
for a suitable constant $c_1>0$. We consider the functional
\begin{equation*}
 \energy E^h(g):=\left\{\begin{aligned}
                  &\frac{1}{h^2}\int\limits_\Omega W(x,\id+h\grad g(x))\ud x&&\text{if }g\in\boundaryset\\
		  &+\infty &&\text{else,}
                 \end{aligned}\right.
\end{equation*}
and assume that there exists a positive constant $c_2$ such that
 \begin{equation}
\label{eq:17}\energy E^h(g)\geq c_2\norm{g}^2_{\fs W^{1,2}(\Omega;\Rrv n)}\qquad\text{for all $h>0$ and }g\in\boundaryset.
 \end{equation}
Then the family $(\energy E_h)$ $\Gamma$-converges with respect to strong convergence in $\fs L^2(\Omega;\Rrv n)$ to the functional
\begin{equation*}
 \energy E_\lin(g):=\left\{\begin{aligned}
     &\frac{1}{h^2}\int\limits_\Omega Q(x,\grad g(x))&&\text{if }g\in\boundaryset\\
     &+\infty&& \text{else.}
   \end{aligned}\right.
\end{equation*}
\end{theorem}
\begin{remark}
  Condition \eqref{eq:17} implies equi-coercivity of the functionals
  $(\mathcal E^h)$ in  $\fs L^2(\Omega;\Rrv n)$. In virtue of
  Proposition~\ref{prop:2}, we see that the combination of the
  non-degeneracy of $W$ and the Dirichlet boundary condition is a
  sufficient condition for \eqref{eq:17}.
\end{remark}
\begin{proof}
Since  the strong topology of $\fs L^2(\Omega;\Rrv n)$ is metrizable,
we can use the sequential characterization of $\Gamma$-convergence.
Thus, we have to prove the following:
\begin{enumerate}
\item (lower bound) for every $g\in\fs L^2(\Omega;\Rrv n)$ and every sequence
  $(g_h)$ converging to $g$ in $\fs L^2(\Omega;\Rrv n)$ it is
  \begin{equation*}
    \liminf\limits_{h\rightarrow 0}\mathcal E^h(g_h)\geq\mathcal E_\lin(g).
  \end{equation*}
\item (recovery sequence) for every $g\in\fs L^2(\Omega;\Rrv n)$ there exists a sequence
  $(g_h)$ converging to $g$ in $\fs L^2(\Omega;\Rrv n)$ such that
  \begin{equation*}
    \lim\limits_{h\rightarrow 0}\mathcal E^h(g_h)=\mathcal E_\lin(g).
  \end{equation*}
\end{enumerate}
\step 1 (Recovery sequence). We only have to consider the case
$g\in\boundaryset$. By assumption the inclusion $\left(\,\fs W^{1,\infty}(\Omega;\Rrv
n)\cap\boundaryset\,\right)\subset\boundaryset$ is dense; thus, there exists a
sequence $(g_h)$ in $\fs W^{1,\infty}(\Omega;\Rrv n)\cap\boundaryset$
that strongly converges to $g$ and that satisfies
\begin{equation}\label{eq:15}
  \esssup\limits_{x\in\Omega}\abs{\grad g_h(x)}\leq
  \frac{1}{\sqrt{h}}\qquad\text{for all }h>0.
\end{equation}
Thus, by \eqref{eq:50} we get
\begin{equation}\label{eq:16}
 \frac{1}{h^2}\int\limits_\Omega W(x,\id+h\grad g_h(x))\ud
 x=\int\limits_\Omega Q(y,\grad g_h(x))+\restterm(x,h\grad g_h(x))\ud x
\end{equation}
with
\begin{equation*}
 \lim\limits_{h\rightarrow 0}\abs{\int\limits_\Omega\restterm(x,h\grad g_h(x))\ud x}=0.
\end{equation*}
Since the quadratic integral functional in \eqref{eq:16} is
continuous with respect to strong convergence and $g_h\rightarrow g$
strongly in $\fs W^{1,2}(\Omega;\Rrv n)$, we see that
\begin{equation*}
  \lim\limits_{h\rightarrow 0}\mathcal E^h(g_h)=\mathcal E_\lin(g).
\end{equation*}

\step 2 (Lower bound). Let $g\in\fs L^2(\Omega;\Rrv n)$ and $(g_h)$ a
sequence converging to $g$ in $\fs L^2(\Omega;\Rrv n)$. We show that
\begin{equation}\label{eq:2}
 \liminf\limits_{h\rightarrow 0}\mathcal E^h(g_h)\geq \mathcal E_\lin(g).
\end{equation}
As usual it is sufficient to consider the case where
\begin{equation*}
 \liminf\limits_{h\rightarrow 0}\mathcal E^h(g_h)=
 \limsup\limits_{h\rightarrow 0}\mathcal E^h(g_h) <+\infty.
\end{equation*}
In this case we can assume without loss of generality that each $g_h$
belongs to the set $\boundaryset$. The
equi-coercivity condition on $(\mathcal E^h)$ (see \eqref{eq:17}) implies that $(g_h)$ is
relatively compact with respect to weak convergence in $\fs
W^{1,2}(\Omega;\Rrv n)$. As a consequence we can pass to a  subsequence
(that we do not relabel) such that 
\begin{equation*}
 g_h\wconv g\qquad\text{weakly in }\fs W^{1,2}(\Omega;\Rrv n).
\end{equation*}
Because $\boundaryset$ is a (seq.) weakly closed subset of $\fs
W^{1,2}(\Omega;\Rrv n)$, we find that $g\in\boundaryset$.

We define the set
\begin{equation*}
 \Omega_h:=\{\,x\in\Omega\,:\,\abs{\grad g_h(x)}\leq h^{-1/2}\,\}
\end{equation*}
and denote the corresponding indicator function by $\chi_h$. Since
$(\grad g_h)$ is a bounded sequence in $\fs L^2(\Omega;\Rrm n)$, we can estimate the measure of $\Omega\setminus\Omega_h$ according to 
\begin{equation}\label{pf:thm:2:1}
 \abs{\Omega\setminus\Omega_h}\leq h\int\limits_{\Omega\setminus\Omega_h}\abs{\grad g_h(x)}^2\ud x\leq h C
\end{equation}
for a suitable constant $C$. Now we establish a lower estimate by
utilizing the expansion of $W$ for points $x\in\Omega_h$ and the non-negativity of $W$ for points $x\in\Omega\setminus\Omega_h$:
\begin{equation*}
 \int\limits_\Omega W(x,\id+h\grad g_h(x))\ud x\geq \int\limits_{\Omega_h}Q(x,\grad g_h(x))+\restterm(x,h\grad g_h(x))\ud x.
\end{equation*}
Since $\abs{h\grad g_h(x)}\leq \sqrt h$ for all $x\in\Omega_h$ the
properties of the quadratic expansion of $W$ lead to 
\begin{equation*}
 \liminf\limits_{h\rightarrow 0}\int\limits_\Omega W(x,\id+h\grad
 g_h(x))\ud x\geq \liminf\limits_{h\rightarrow
   0}\int\limits_{\Omega}Q(x,\chi_h(x)\grad g_h(x))\ud x.
\end{equation*}
Since $\grad g_h\wconv \grad g$ weakly in $\fs L^2(\Omega;\Rrm n)$
and $\chi_h\rightarrow 1$ boundedly in measure (due to
\eqref{pf:thm:2:1}), we find that
\begin{equation*}
  \chi_h\grad g_h\wconv \grad g\qquad\text{weakly in }\fs
  L^2(\Omega;\Rrm n)
\end{equation*}
and the sequentially weak lower semi-continuity of the quadratic integral
functional leads to the desired lower bound estimate \eqref{eq:2}.
 \end{proof}

\section{Homogenization and linearization commute}\label{sec:hlc}
In this section we prove that the diagram \eqref{diagram} in Theorem
\ref{thm:4} commutes. We have to show that the
$\Gamma$-convergence statements
\begin{alignat*}{6}
 &(1)\qquad\energy I^{\varepsilon,h}&&\gammaconv \energy I_\lin^\varepsilon &\qquad\qquad&(2)&\qquad&\energy I^{\varepsilon,h}&&\gammaconv \energy I_\hom^h\\
 &(3)\qquad\energy I^{\varepsilon}_\lin&&\gammaconv \energy I^0 &\qquad\qquad&(4)&\qquad&\energy I^{h}_\hom&&\gammaconv \energy I^0
\end{alignat*}
hold w.r.t. strong convergence in $\fs L^2(\Omega;\Rrv n)$.

First, we discuss the $\Gamma$-limits corresponding to linearization,
i.e. (1) and (4). To this end we recall that Proposition~\ref{prop:2}
 implies that
\begin{equation}\label{eq:coercive}
\min\{\,\energy I^{\varepsilon,h}(g),\,\energy I^{h}_\hom(g)\,\}\geq
c\norm{g}^2_{\fs W^{1,2}(\Omega;\Rrv n)}
\end{equation}
for all $g\in\boundaryset$ and $h,\varepsilon>0$. As a consequence we
see that convergence (1) directly follows by applying
Theorem~\ref{thm:2} to the sequence $(\mathcal I^{\varepsilon,h})_h$.

For the justification of convergence (4), we first apply Theorem~\ref{thm:1} and see that
$W\mc$ admits a quadratic Taylor expansion. It is well-known that
the homogenization $Q_\hom\super 1$ of a quadratic integrand $Q\in\mathcal Q$ is a
quadratic form over $\Rrv n$, and therefore satisfies the condition of Theorem~\ref{thm:2}. Moreover,
Proposition~\ref{prop:2} proves that the sequence $(\energy I^h_\hom)$
is equi-coercive and $(4)$ follows by applying Theorem~\ref{thm:2}.

The $\Gamma$-limits corresponding to homogenization, i.e. (2) and (3),
follow by standard results. More precisely, convergence (3) is a
classical result in convex homogenization (cf. e.g. \cite{Francfort1986},
\cite{Olejnik1984} and note that these results can be restated in the
language of $\Gamma$-convergence). The non-convex homogenization, convergence (2), can be justified by the following result:

\begin{theorem}[A.~Braides \cite{Braides1985}, S.~M\"uller
  \cite{Mueller1987}]\label{thm:homog}
Let   $\Omega$ be a Lipschitz domain in $\Rrv n$. Suppose that $W:\Rrv
n{\times}\Rrm n\rightarrow \Rr$ is a
Carath\'eodory function that is $Y$-periodic in its first variable and
that satisfies the growth and  coercivity condition (W1) with
$p\in(1,\infty)$. 
We consider the functionals
\begin{equation*}
 \energy E^\varepsilon(u):=\left\{\begin{aligned}
                    &\int\limits_\Omega W(x/\varepsilon,\grad u(x))\ud x&\qquad&\text{if }u\in\fs W^{1,p}(\Omega;\Rrv n)\\
		    &+\infty&&\text{else,}
                   \end{aligned}\right.
\end{equation*}
and
\begin{equation*}
 \energy E_\hom(u):=\left\{\begin{aligned}
                    &\int_\Omega W\mc(\grad u(x))\ud x&\qquad&\text{if }u\in\fs W^{1,p}(\Omega;\Rrv n)\\
		    &+\infty&&\text{else.}
                   \end{aligned}\right.
\end{equation*}
Then $(\mathcal E^\varepsilon)$ $\Gamma$-converges to $\mathcal
E_\hom$ with respect to strong convergence in $\fs L^p(\Omega;\Rrv n)$
as $\varepsilon\rightarrow 0$.
\end{theorem}
The previous theorem can be adapted to the boundary conditions
considered in diagram \eqref{diagram}. More precisely, one can show that
\begin{quote}
 for every $u\in\fs W^{1,p}(\Omega;\Rrv n)$ there is a sequence $(u_\varepsilon)$ such that $u_\varepsilon\conv u$ strongly in $\fs L^p(\Omega;\Rrv n)$, $u_\varepsilon=u$ on $\partial\Omega$  and $\lim\limits_{\varepsilon\rightarrow 0}\energy E^\varepsilon(u_\varepsilon)=\energy E_\hom(u)$.
\end{quote}
This can be seen by a standard gluing argument (cf. e.g.
\cite{Mueller1987,Braides1998}). In particular, we can apply this
modification in the case $p\geq 2$ and to limiting deformations
$u\in\fs W^{1,2}(\Omega;\Rrv n)$ satisfying $u(x)=x$ on $\gamma$. Hence, the theorem implies that $(\mathcal
I^{\varepsilon,h})_\varepsilon$ $\Gamma$-converges to $\mathcal
I^h_\hom$ with respect to the strong topology in $\fs L^2(\Omega;\Rrv
n)$. 

In conclusion we see that both paths (1), (3) and (2), (4) lead to the
same limiting functional $\energy I^0$ and the diagram commutes.

\section{Failure of commutativity}\label{sec:non-comm-expans-3}
In this section we present two examples illustrating that
homogenization and linearization cannot be interchanged in general. In the first
example we consider energy densities of class $\mathcal
W(a,p)$ and show that the commutativity might fail if the expansion
is not centered at identity.
Secondly, we elaborate on the importance of condition
(W2) and (W3), which imply that the material has a single energy well
at $\SO n$ with quadratic growth. As an example we discuss  a perforated composite with a
prestressed component and give an indirect argument that
homogenization and linearization do not commute at identity
--- although the homogenized energy density is stress free at identity
and satisfies  $W\mc(\id+G)=O(|G|^2)$.

\subsection{Counterexample I: Non-commutativity for expansions away
  from $\SO n$}

In this section we argue that in general the commutability of  homogenization and
linearization does not hold for expansions centered at $F\not\in\SO n$. Roughly
speaking, the reason is the following: If $F\not\in\SO n$ we may find
a nonlinear material $W\in\mathcal W(a,p)$ with $W\mc(F)<W\super
1_\hom(F)$, which means that long-wave
oscillations with period $k\varepsilon$, $k\geq 2$, lead to limiting energies that are lower than those
obtained by ``one-cell'' oscillations with period $\varepsilon$. On the other side, the
linearized energy and its homogenization is always quadratic, and due
to the non-degeneracy of the
material stable in a neighborhood of $\SO n$. Thus, low energy states related to long-wave
oscillations are ignored by the homogenization of the linearized energy;
and therefore homogenization and linearization do not commute for
$F\not\in\SO n$ in general.

In the rest of this section we present an implicit, but rigorous
formulation of the idea above by varying an example
introduced by the first author in \cite{Mueller1987}.
Let $W_0:\,\Rrm
2\rightarrow[0,\infty)$ be a frame indifferent integrand of class
$C^3$. We suppose that $W_0$ satisfies the local Lipschitz- and growth
condition $(W1)$ with $p=2$, and
\begin{equation}\label{eq:21}
    \frac{1}{c}\dist^2(F,\SO 2)\leq W_0(F)\leq c\dist^2(F,\SO
    2)\qquad\text{for all }F\in\Rrm 2.
\end{equation}
We consider the periodic stored energy function
\begin{equation*}
  W(y,F):=\big(\,\chi(y)+\alpha(1-\chi(y))\,\big)W_0(F),
\end{equation*}
where $\alpha$ is a small positive parameter and $\chi$ denotes the indicator function of the periodic pattern
\begin{equation*}
  P:=\Big\{ (x_1,x_2)\in\Rrv 2\,:\,\exists k\in\Zz\,\text{ such that }\,x_2\in[k,k+\tfrac{1}{2})\,\Big\}  
\end{equation*}
Thus, $W$ describes a composite material with a layered microstructure consisting
of a stiff component on $P$ and a soft matrix with stiffness
$\alpha$. 

As a main result we prove a necessary condition for the property that
linearization and homogenization commute for all expansions in an neighborhood of $\id$:
\begin{proposition}\label{prop:1}
Let $U$ be an open neighborhood of $\id$. Suppose that for each expansion at $F\in U$ homogenization and linearization commute in the
following sense: There exists $\sigma_F\in\Rrm 2$ such that
\begin{equation*}
\forall G\in\Rrm 2\,:\qquad  W\mc(F+G)=W\mc(F)+\dprod{\sigma_F}{G} + Q^F_\hom(G)+o(\abs{G}^2)
\end{equation*}
where $Q^F_\hom$ denotes the homogenization of the quadratic integrand
\begin{equation*}
  Q^F(y,G):=\frac{1}{2}D^2W(y,F)(G,G).
\end{equation*}
Then 
\begin{equation*}
  \alpha\geq c(W_0,U)
\end{equation*}
for a positive constant $c(W_0,U)>0$ that only depends on
$W_0$ and the neighborhood $U$.
\end{proposition}
\begin{remark}
  \label{sec:non-comm-expans-2}
  Theorem \ref{thm:1} implies that the expansion holds for $F=\id$
  with $\sigma_F=0$.  
\end{remark}
\begin{remark}\label{sec:non-comm-expans-1}
  The result can be read as follows: For any open neighborhood $U$ of $\id$ we
  can find a material (e.g. by choosing $\alpha$ sufficiently small)
  such that homogenization and linearization do not commute for all
  expansions centered in $U$.  
\end{remark}

The result follows from a slightly stronger statement that we prove below. In particular,
it is sufficient to study the response of the homogenized material to compressions of the form
\begin{equation*}
  x\mapsto F_\delta x,\qquad F_\delta:=\id-\delta(e_1{\otimes}e_1),\qquad \delta\geq 0
\end{equation*}
For future reference we set
\begin{equation*}
  Q^\delta(y,G):=\frac{1}{2}D^2W(y,F_\delta)(G,G)
\end{equation*}
and let $Q^\delta_\hom$ denote its homogenization.

In \cite{Mueller1987} it was shown that if $\delta>0$ and $\alpha$ is
sufficiently small, the stiff part of the material, which resembles an
ensemble of aligned rods, starts to buckle and allows to bound the
limiting energy linearly in  $\alpha$. In contrast to this, we prove that the
non-degeneracy of $W$ implies that  $Q^\delta_\hom$ is stable for
small $\delta$. Both observations can be quantified as follows:
\begin{lemma}\label{lem:3}
  There exist positive constants $\delta_0$ and $c_0$ (depending only
  on $W_0$) with the
  following properties:
  For all $0\leq \delta\leq\delta_0$ and $\alpha>0$ we have
  \begin{align}\label{eq:25}
    Q^\delta_{\hom}(e_1{\otimes}e_1)&\geq \frac{1}{c_0},\\
    \label{eq:22}
    W\mc(F_\delta)&\leq c_0\alpha.
  \end{align}
\end{lemma}

We postpone the proof to the end of this section and present a quite
immediate consequence of the lemma that already implies the validity
of Proposition~\ref{prop:1}.

\begin{lemma}\label{lem:6}
Define the map
  \begin{equation*}
    f:\,\Rr\rightarrow \Rr,\qquad f(\delta):=W\mc(F_\delta).
  \end{equation*}
  Then for almost every $\delta\in[0,\delta_0]$ the map $f$
  admits a quadratic Taylor expansion of the form 
  \begin{equation}\label{eq:28}
    f(\delta+\lambda)=f(\delta)+\sigma_\delta\lambda+q_\delta\lambda^2+o(\lambda^2)
  \end{equation}
  for suitable numbers $\sigma_\delta\in\Rr$ and $q_\delta\geq 0$.
  \begin{enumerate}
    \item If $\delta=0$, the expansion \eqref{eq:28} is valid for
      \begin{equation*}
        \sigma_0=0\qquad\text{and}\qquad q_0=Q^0_\hom(e_1{\otimes}e_1),
      \end{equation*}
      i.e. linearization and homogenization commute at $F_0=\id$.
      \item Suppose that linearization and homogenization commute at
        $F_\delta$ for all a.e.  $\delta\in[0,\delta_0]$, i.e.
        \begin{equation*}
          q_\delta=Q^\delta_\hom(e_1{\otimes}e_1)\qquad\text{for a.e. }\delta\in[0,\delta_0].
        \end{equation*}
        Then
        \begin{equation*}
          \alpha\geq \frac{\delta_0^2}{2c_0^2}.
        \end{equation*}
  \end{enumerate}
\end{lemma}
\begin{remark}
  \label{sec:non-comm-expans}
  In the lemma above the constants $\delta_0$ and $c_0$ only depend on
  $W_0$. Thus, if the soft matrix of the composite material is
  sufficiently weak, i.e. $\alpha<\!\!< 1$, statement (2) suggests
  that homogenization and linearization may only commute in a very small
  neighborhood of $\id$.
\end{remark}
\begin{proof}[Proof of Lemma~\ref{lem:6}]
  By Lemma~\ref{lem:2} the map  $W\mc:\,\Rrm 2\rightarrow\Rr$ is
  continuous and quasiconvex, and therefore rank-one convex. Because
  $F_\delta-F_\lambda$ is always a rank-one matrix, we deduce that $f$
  is convex and continuous. Thus, Aleksandrov's Theorem implies that
  $f$ admits an expansion of the form \eqref{eq:28} for a.e.
  $\delta\in[0,\delta_0]$. Statement (1) is a direct consequence of Theorem~\ref{thm:1}
  applied to $G=-\lambda(e_1{\otimes}e_1)$.

  We prove (2). By assumption we have
  \begin{equation*}
    f^{\prime\prime}(\delta)=Q^\delta_\hom(e_1{\otimes}e_1)
  \end{equation*}
  for a.e. $\delta$. We apply  \eqref{eq:25} from  Lemma~\ref{lem:3} and find that $f^{\prime\prime}\geq c_0^{-1}$
  almost everywhere. Thus, the map
  \begin{equation*}
    g(\delta):=f(\delta)-\frac{1}{2c_0}\delta^2
  \end{equation*}
  is convex and by part (1) we have $g^\prime(0)=f^\prime(0)=0$. Thus
  $g$ attains its minimum at $0$ with $g(0)=0$. By \eqref{eq:22} we get
  \begin{equation*}
    \frac{1}{2c_0}\delta^2\leq f(\delta)\leq c_0\alpha\qquad\text{for
      all }\delta\in[0,\delta_0]
  \end{equation*}
  and the estimate for $\alpha$ follows.
 \end{proof}

\begin{proof}[Proof of Lemma~\ref{lem:3}]
  In the sequel $c^\prime, c^{\prime\prime}$ denote positive constants that may change
  from line to line, but can be chosen only depending on $\delta_0$
  and $W_0$.

  \step 1 We prove \eqref{eq:25}. Let $Q^\delta_0$ denote the quadratic term
  in the expansion of $W_0$ at $F_\delta\in\Rrm 2$, i.e.
  \begin{equation*}
    Q_0^\delta(G):=\frac{1}{2}D^2W_0(F_\delta)(G,G).
  \end{equation*}
  Because $F_0=\id$ and $W_0$ is frame-indifferent and non-degenerate, the estimate
  \begin{equation*}
    Q_0^0(G)\geq \frac{1}{c^\prime} \abs{\sym G}^2
  \end{equation*}
  holds for all $G\in\Rrm 2$.
  The map $W_0$ is of class $C^3$, and therefore  a standard perturbation argument
  shows that
  \def\asym{\mathop{\operatorname{skw}}}
  \begin{equation*}
    Q_0^\delta(G)\geq \frac{1}{c^\prime}\abs{\sym
      G}^2-c^{\prime\prime}\delta\abs{\asym G}^2
  \end{equation*}
  for all $\delta\in[0,\delta_0]$ where $\delta_0$ denotes a small positive
  constant that only depends on $W_0$.

  Let $Y_0:=[0,1){\times}[0,\tfrac{1}{2})$ denote the reference domain
  of the stiff component  and consider the function space 
  \begin{equation*}
    V:=\Big\{\,\psi\in W^{1,2}(Y_0;\Rrv
    2)\,:\,\psi(0,x_2)=\psi(1,x_2),\,\int_{Y_0}\psi(x)\ud x=0\,\Big\}.
  \end{equation*}
  Note that for all $\psi\in V$ a
  Korn inequality of the form
  \begin{equation*}
    \int_{Y_0}\abs{\asym \grad\psi}^2\ud y\leq
    c^\prime\int_{Y_0}\abs{\sym\grad \psi}^2\ud y
  \end{equation*}
  holds. As a consequence we deduce that
  \begin{equation*}
   m(\delta_0):=\inf_{\delta\in(0,\delta_0)}\inf_{\psi\in V}\int_{Y_0}Q_0^\delta(e_1{\otimes}e_1+\grad\psi(x))\ud y>0
  \end{equation*}
  provided  $\delta_0$ is sufficiently small. 

  Next, we are going to show that the estimate above yields a lower bound for $Q^\delta_\hom(G)$: Let
  $\delta\in[0,\delta_0]$ and $\psi\in
  W^{1,2}_\per(Y;\Rrv 2)$. Then 
  \begin{equation*}
    \int_YQ^\delta\big(y,(e_1{\otimes}e_1)+\grad\psi\big)\ud y\geq
    \int_{Y_0}Q_0^\delta\big((e_1{\otimes}e_1)+\grad\psi\big)\ud
    y\geq m(\delta_0)>0
  \end{equation*}
  and \eqref{eq:25} is valid for all  $c_0\geq m(\delta_0)^{-1}$.

  \step 2 We prove \eqref{eq:22} by constructing a sequence
  $(u_k)_{k\in\Nn}$ such that
  \begin{equation}\label{eq:23}
    \liminf\limits_{k\rightarrow\infty}\int_Y
    W\left(kx,\grad u_k(x)\right)\ud x\leq c_0\alpha
  \end{equation}
  and $u_k\wconv F_\delta x+\varphi(x)$ weakly in
  $W^{1,2}(Y;\Rrv 2)$ with $\varphi\in W^{1,2}_{\per}(Y;\Rrv 2)$.
  Because   the functional
  \begin{equation*}
    W^{1,2}(Y;\Rrv 2)\ni u\mapsto \int_Y
    W\left(kx,\grad u(x)\right)\ud x
  \end{equation*}
  $\Gamma$-converges in the strong topology of $L^{2}(Y;\Rrv 2)$ to 
  \begin{equation*}
    W^{1,2}(Y;\Rrv 2)\ni u\mapsto \int_Y W\mc(\grad u(x))\ud x
  \end{equation*}
  as $k\rightarrow \infty$,
  \eqref{eq:23} implies that
  \begin{equation}\label{eq:24}
    \int_Y W\mc(F_\delta+\grad\varphi(x))\ud x\leq\liminf\limits_{k\rightarrow\infty}\int_Y
    W\left(kx,\grad u_k(x)\right)\ud x\leq c_0\alpha.
  \end{equation}
  The homogenized integrand $W\mc$ is quasiconvex, and therefore the
  left hand side is bounded from below by $W\mc(F_\delta)$.
  
  It remains to construct the sequence $(u_k)$. We follow the idea in
  \cite{Mueller1987}: The microstructure of the material can be regarded as an
  ensemble of thin rods, aligned to the $e_1$-direction and embedded in a soft matrix. The
  subsequent construction realizes the compression $x\mapsto F_\delta
  x$ (more precisely, a periodic variation) by bending each
  rod. In this way, locally the deformation in the stiff component is close to a rigid motion, while large strains only evolve in the soft
  matrix.
  
  For the precise construction, let $v_\delta\in C^\infty([0,1];\Rrv
  2)$ denote a curve,  parametrized by arc-length and rendering an arc with
  length $1$ connecting the points $(0,0)$ and $(1-\delta,0)$, i.e.
  \begin{equation*}
    v_\delta(0)=(0,0),\qquad v_\delta(1)=(1-\delta,0),\qquad
    \abs{v^\prime_\delta(x_1)}=1,\qquad \abs{v^{\prime\prime}_\delta(x_1)}=\kappa_\delta
  \end{equation*}
  for some $\kappa_\delta>0$. Let $n_\delta\in C^\infty([0,1];\Rrv 2)$
  denote a normal field associated to the curve, so that the
  map $x_1\mapsto R_\delta(x_1):=v_\delta^\prime(x_1){\otimes}
  e_1+n_\delta(x_1){\otimes}e_2$ only takes values in $\SO 2$.

  For $x\in Y$
  define
  \begin{equation*}
    u_k(x):=\left\{
      \begin{aligned}
        &x_2e_2+v_\delta(x_1)+(x_2-\tfrac{i}{k})\,(n_\delta(x_1)-e_2)&&\text{if
        }&&x_2\in[\tfrac{i}{k},\tfrac{2i+1}{2k})\\
        &&&&&\text{for some }i\in\Nn_0,\\
        &x_2e_2+v_\delta(x_1)+(\tfrac{i+1}{k}-x_2)(n_\delta(x_1)-e_2)&&\text{if
        }&&x_2\in[\tfrac{2i+1}{2k},\tfrac{i+1}{k})\\
        &&&&&\text{for some }i\in\Nn_0.
      \end{aligned}
    \right.
  \end{equation*}
  By construction we have  $u_k\in W^{1,2}(Y;\Rrv 2)$ and $u_k\rightarrow v_\delta+x_2e_2$ uniformly. The
  sequence $(u_k)$ is bounded in $W^{1,2}(Y;\Rrv 2)$ and therefore
  \begin{equation*}
    u_k\wconv F_\delta x+\varphi(x)\qquad\text{weakly in
    }W^{1,2}(Y;\Rrv 2)
  \end{equation*}
  where $\varphi(x):=v_\delta(x_1)+x_2e_2-F_\delta x$. Furthermore, it is easy to
  check that $\varphi$ can be identified with a map in  $W^{1,2}_\per(Y;\Rrv 2)$. 
  
$\grad
  u_n$ is close to a rotation on the strong part of the material. More precisely, we have
  \begin{equation*}
    \dist^2(\grad u_k(x),\SO 2)\leq \left(\frac{\kappa_\delta^2
    }{4}\right)k^{-2}\qquad\text{for all }x_2\in[\tfrac{i}{k},\tfrac{2i+1}{2k}),
  \end{equation*}
  while on the soft part the deformation is at least bounded, i.e.
  \begin{equation*}
    \abs{\grad u_k(x)}^2\leq c^\prime\qquad\text{for all }x_2\in[\tfrac{2i+1}{2k},\tfrac{i+1}{k}).
  \end{equation*}
  Thus, by using the definition of $W$ and the growth condition \eqref{eq:21}, we see that
  \begin{multline*}
    \int_Y W(kx,\grad u_k(x))\ud x\leq c^\prime
    \sum_{i=0}^{k-1}\Bigg(\,\int_0^1\int_{i/k}^{(2i+1)/2k}\dist^2(\grad
    u_k,\SO 2)\ud x_2\\
    \qquad\qquad\qquad\qquad+\alpha\int_0^1\int_{(2i+1)/2k}^{(i+1)/k}\abs{\grad
      u_k}^2{+}2\ud x_2\ud x_1\Bigg)\\
  \leq c^\prime(k^{-2}+\alpha)
  \end{multline*}
  and \eqref{eq:24} follows as $k\rightarrow\infty$.
 \end{proof}

\subsection{Counterexample II:  Non-commutativity at identity for a
  perforated, prestressed composite material}\label{sec:count-ii:-non}
In this section we construct an example in dimension $3$ for which the commutativity of linearization and
homogenization at identity fails, although the identity is a natural state
of the homogenized material. The material under consideration is a
composite with a prestressed component that violates property (W2).
Furthermore, the material is perforated and the non-degeneracy condition (W3) is
violated in $x_3$-direction.

A commutativity statement in the form of Theorem~\ref{thm:1}, comparing the homogenized and the spatially
heterogeneous energy density, is not suited for a prestressed material
due to the lack of a stress free reference configuration. Therefore, we introduce a
weaker property that is motivated by the following observation:
\begin{corollary}\label{cor:1}
  Let $W\in\mathcal W(a,p)$. For $k\in\Nn$ let
  \begin{equation*}
    W\super k_\hom(F):=\inf\left\{\,\frac{1}{k^n}\int\limits_{kY} W(y, F+\grad\varphi(y))\ud y\,:\,\varphi\in\fs W^{1,2}_\per(kY;\Rrv n)\,\right\}
  \end{equation*}
  denote the $k$-cell homogenization of $W$. Then
  \begin{equation}\label{eq:38}
\lim\limits_{h\downarrow 0}\frac{1}{h^2}W\super
      k_\hom(\id+h G)=\lim\limits_{h\downarrow 0}\frac{1}{h^2}W\mc(\id+h G).
  \end{equation}
 for all $k\in\Nn$
  and all $G\in\Rrm n$.
\end{corollary}
\begin{proof}
  The statement is implicit in the proof of Theorem~\ref{thm:1}:
  Indeed, in Step~1 we showed that $\limsup\limits_{h\downarrow 0}\frac{1}{h^2}W\super
  1_\hom(\id+h G)\leq Q\super 1_\hom(G)$, while in Step~2 we proved
  the inequality $\liminf\limits_{h\downarrow 0}\frac{1}{h^2}W\mc(\id+h
  G)\geq Q\super 1_\hom(G)$. Since $W\mc\leq W\super k_\hom\leq
  W\super 1_\hom$ we can draw the conclusion.
\end{proof}
The corollary indicates that in situations when commutativity holds,
close to identity relaxation over a single cell is sufficient for
homogenization. Proposition~\ref{prop:1} shows that if we replace
$\id$ by a matrix $F\not\in\SO n$, identity \eqref{eq:38} fails for
some materials of class $\mathcal W(a,p)$. This suggests to introduce the
following property: We say that homogenization and linearization in
      direction $G\in\Rrm n$ commute at $\id$, if
\begin{equation}\tag{$C_{G}$}\label{eq:37}
 \left.\begin{aligned}
    &\text{for all $k\in\Nn$ there exist $\sigma\super k,\sigma\in\Rrm n$ such that}\\
    &\lim\limits_{h\downarrow 0}\frac{1}{h^2}\left(W\super
      k_\hom(\id+h G)-W\super k_\hom(\id)-h\iprod{\sigma\super k}{G}\right)\\
    &\qquad\qquad=\lim\limits_{h\downarrow 0}\frac{1}{h^2}\left(W\mc(\id+h G)-W\mc(\id)-h\iprod{\sigma}{G}\right)
  \end{aligned}\right\}
\end{equation}
Note that we distinguish between the
directions $+G$ and $-G$. By Corollary \ref{cor:1} materials of class $\mathcal W(a,p)$ satisfy \eqref{eq:37} for all
$G\in\Rrm n$. 

In the following, we present an example in dimension $n=3$ that
violates \eqref{eq:37}. Let us first describe the geometry and the
energy density of the composite. The composite's geometry is given by two subsets of the
reference cell $Y:=[0,1)^3$: 
\begin{gather*}
Y_0:=[0,1)^2{\times}[0,\tfrac{1}{2})\qquad\text{and}\qquad
Y_\rho:=[0,1){\times}B_\rho,\\\text{where
}B_\rho:=\left\{(y_2,y_3)\,:\,(y_2-\tfrac{1}{2})^2+(y_3-\tfrac{3}{4})^2\leq
  \rho^2\,\right\}, \qquad 0<\rho\ll 1.
\end{gather*}
The set $Y_\rho$ is a cylinder with center line
$\{(y_1,\tfrac{1}{2},\tfrac{3}{4})\,:\,y_1\in[0,1)\,\}$ and small radius $\rho$.
Let $W_0$ be of class $\mathcal W(a,2)$. We
suppose that $W_0$ is frame-indifferent, i.e. $W_0(RF)=W_0(F)$ for all
matrices $F\in\Rrm 3$ and rotations $R\in\SO 3$. Set
$S:=(\id+s(e_1{\otimes}e_1))^{-1}$ with $0<s<\tfrac{1}{2}$. We denote by $W:\,\Rrv
3{\times}\Rrm 3\to[0,\infty)$ the $Y$-periodic energy density 
\begin{equation*}
  W(y,F):=
  \begin{cases}
    W_0(F)&\text{if }y\in Y_0,\\
    W_0(F\,S)&\text{if }y\in Y_\rho,\\
    0&\text{if }y\in Y\setminus(Y_0\cup Y_\rho).
  \end{cases}
\end{equation*}
Clearly, for the material occupying $Y_0$ every rotation is a natural state, while
on $Y_\rho$ the material is stress free only for matrices $F=RS^{-1}$
with $R\in\SO 3$. 

\begin{lemma}\label{lem:7}
  The following properties hold:
  \begin{enumerate}
  \item[a)] $W$ is frame-indifferent, has quadratic growth at infinity and
    is non-degenerate in $x_1$- and $x_2$-direction in the sense that
    \begin{equation}\tag{W3'}\label{eq:39}
      \begin{aligned}
        &\int_Y W(y,F+\nabla\psi(y))\ud
        y\geq c'\dist^2(F,SO(3))\\
        &\qquad\text{for all }\psi\in W^{1,2}_\per(Y;\Rrv 3)\text{ with
        }\int_0^1\int_0^1\partial_3\psi\ud y_1\ud y_2=0.
      \end{aligned}
    \end{equation}
  \item[b)] For all $k\in\Nn$ we have $\inf_{F\in\Rrm 3}W\super
    k_\hom(F)>0$.
  \item[c)] (Natural state). 
    \begin{equation*}
      W\mc(\id)=0=\min\limits_{F\in\Rrm
        3}W\mc(F).
    \end{equation*}
  \item[d)] (Expansion at $\id$). For all $G\in\Rrm 3$ we have
    \begin{equation*}
      \limsup\limits_{h\to
      0}\frac{1}{h^2}W\mc(\id+hG)<\infty.
    \end{equation*}
  \item[e)] (Failure of \eqref{eq:37}). There exists $G\in\Rrm 3$ such that for all $k\in\Nn$ and
  all $\sigma\super k\in\Rrm 3$ with
  \begin{equation}\label{eq:47}
    \limsup\limits_{h\downarrow 0}\frac{1}{h^2}\left|W\super k_\hom(\id+h
      G)-W\super
      k_\hom(\id)-h\iprod{\sigma\super k}{G}\right|<\infty
  \end{equation}
  we have 
  \begin{multline}\label{eq:36}
    \liminf\limits_{h\downarrow 0}\frac{1}{h^2}\left(W\super k_\hom(\id+h
      G)-W\super
      k_\hom(\id)-h\iprod{\sigma\super k}{G}\right)\\
    >0=\lim\limits_{h\downarrow
      0}\frac{1}{h^2}W\mc(\id+hG).
  \end{multline}
  \end{enumerate}
\end{lemma}

In a nutshell, the idea behind the construction is the following: Close to identity and for finite $k$, the prestressed component contributes a
positive amount of energy and yields a shift of the natural state to
some $F\not\in \SO 3$. When $k$ increases, the aspect ratio
(thickness / length) of the cylindrical component and (as a consequence) the contributed
elastic energy decreases. We show that in the limit $k\to\infty$ the energy
contribution of the $Y_\rho$-component vanishes. Furthermore, we argue
that $W\mc$ vanishes for short maps.

\begin{proof}[Proof of Lemma~\ref{lem:7}]
Let us introduce the functionals
  \begin{align*}
    I\super k(\varphi;F)&:=\frac{1}{k^3}\int_{kY}W(y,F+\nabla\varphi))\ud y,\\
    I\super k_0(\varphi;F)&:=\frac{1}{k^2}\int_{Z\super k_0}W_0(F+\nabla\varphi)\ud
    y,\qquad &&Z\super k_0:=(0,k)^2{\times}(0,\tfrac{1}{2}),\\
    I\super k_\rho(\varphi;F)&:=\frac{1}{k}\int_{Z\super k_\rho}W_0\Big((F+\nabla\varphi)S\Big)\ud
    y,\qquad &&Z\super k_\rho:=(0,k){\times}B_\rho.
\end{align*}  
and the function spaces
\begin{align*}
  X\super k&:=W^{1,p}_\per(kY;\Rrv 3),\\
  X\super k_0&:=\left\{\,\varphi\big\vert_{Z\super k_0}\,:\,\varphi\in X\super k\,\right\},\\
  X\super k_\rho&:=\left\{\,\varphi\big\vert_{Z\super k_\rho}\,:\,\psi\in X\super k\,\right\}.
\end{align*}
Note that by definition we have $W\super k_\hom(F)=\inf_{\varphi\in
  X\super k}I\super k(\varphi;F)$.

The proof is divided in several steps. In Step~1 we argue that since
the components $Y_0$ and $Y_\rho$ are separated, we can split $I\super
k$ into $I\super k_0$ and $I\super k_\rho$. In Step~2 we argue that
$I\super k_\rho(\cdot;F)$ is positive for $F$ close to $\id$ and show
that its contribution vanishes as $k\to\infty$. In Step~3 we analyze
the behavior of $I\super k_0$. In particular, for the asymptotic behavior $k\to\infty$ we appeal to a
$\Gamma$-convergence result for membranes (see \cite{Raoult1995,Conti2003}). In Step~4 we draw the conclusion.
  
\step 1 (Splitting of $I\super k$). We claim that  for all
$(\varphi,\psi)\in X\super k_0{\times}X\super k_\rho$ there exists
$\Phi\in X\super k$ such that
\begin{equation}\label{eq:30}
  I\super k(\Phi;F)=I\super k_0(\varphi;F)+I\super k_\rho(\psi;F)
\end{equation}
and
\begin{equation}\label{eq:31}
  \inf_{\Phi\in X\super k}I\super k(\Phi;F)= \inf_{\varphi\in
    X\super k_0}I\super k_0(\varphi;F)+\inf_{\psi\in X\super
    k_\rho}I\super k_\rho(\psi;F).
\end{equation}
Statement \eqref{eq:30} easily follows from a periodicity property
of the composite's microstructure. Namely, we can rewrite the subset
of the multi-cell $kY$ occupied by the
$Y_0$-material as a union of $k$ translations of the set
$Z\super k_0$; and the subset occupied by the
$Y_\rho$-material as $k^2$ translations of the set $Z\super k_\rho$:
\begin{align*}
  Y\super
  k_0&:=\bigcup_{\xi\in\Zz^3\cap[0,k)^3}(\xi+Y_0)=\bigcup_{i\in\Zz\cap[0,k)}\left((0,0,i)+Z\super
    k_0\right),\\
  Y\super
  k_\rho&:=\bigcup_{\xi\in\Zz^3\cap[0,k)^3}(\xi+Y_\rho)=\bigcup_{i,j\in\Zz\cap[0,k)}\left((0,i,j)+Z\super
  k_\rho\right).
\end{align*}
For $y\in Y\super k_0\cup Y\super k_\rho$ set
\begin{equation*}
  \widetilde\Phi(y):=
  \begin{cases}
    \varphi(y)&\text{if }y\in (0,0,i)+ Z\super k_0\text{ for some
    }i\in\Zz\\
    \psi(y)&\text{if }y\in(0,i,j)+ Z\super k_\rho\text{ for some
    }i,j\in\Zz
  \end{cases}
\end{equation*}
Since $\overline{Y\super k_0}\cap\overline{Y\super k_\rho}=\emptyset$
we can extend $\widetilde\Phi$ to a $kY$-periodic map  $\Phi\in
X\super k$. Since $W(y,\cdot)\equiv 0$ whenever $y\in
(kY)\setminus(Y\super k_0\cup Y\super k_\rho)$ identity \eqref{eq:30}
follows. \eqref{eq:31} follows from \eqref{eq:30} and the observation
that ``$\geq$'' trivially holds in \eqref{eq:31}, since the
minimization problems on the  r.~h.~s. are
less constrained.

\step 2 (Estimates for $I\super k_\rho$). Set
\begin{equation*}
  m\super k_\rho(F):=\inf\limits_{\psi\in X\super k_\rho}I\super k_\rho(\psi;F).
\end{equation*}
We claim that 
\begin{align}\label{eq:32}
  \forall k\in\Nn,\,G\in\Rrm 3\,&:\,
\liminf\limits_{h\to 0}m\super k_\rho(\id +h G)>0\\
\label{eq:35}\forall G\in\Rrm 3\,&:\,\limsup\limits_{k\to\infty}m\super
k_\rho(\id+ h G)=
\begin{cases}
  0&\text{if }Ge_1=0\\
  O(h^2)&\text{else.}
\end{cases}
\end{align}
\medskip

Proof of \eqref{eq:32} by contradiction: Suppose there exists a
vanishing sequence of positive numbers $(h_j)_{j\in\Nn}$ and a sequence $(\psi_j)\subset X\super
k_\rho$ such that
\begin{equation}\label{eq:49}
  \liminf\limits_{j\to \infty} I\super k_\rho(\psi_j;\id+ h_jG)=0.
\end{equation}
 Consider the deformation $u_j(x):=(\id+h_jG)x+\psi_j(x)$ and the
 mapping $v_j(z):=u_j(Sz)$ defined for $z\in S^{-1}Z\super k_\rho$.
 Since $\nabla v_j(z)=\nabla u_j(Sz)S$, a change of variables yields
 \begin{equation*}
   \frac{1}{(1+s)k}\int_{S^{-1}Z\super
    k_\rho}W_0(\nabla v_j)\ud z=\frac{1}{k}\int_{Z\super
    k_\rho}W_0(\nabla u_j(x)S)\ud x=I\super
  k_\rho(\psi_j;\id+ h_jG).
 \end{equation*}
Thus, by \eqref{eq:49} we have (up to a subsequence)
\begin{equation}\label{eq:40}
  \frac{1}{(1+s)k}\int_{S^{-1}Z\super
    k_\rho}W_0(\nabla v_j)\ud z\to 0\qquad\text{(as $j\to\infty$).}
\end{equation}
In combination with the non-degeneracy condition (W3) and by appealing to geometric
rigidity (see Theorem~\ref{thm:rigidity}) we find (after passing
to a further subsequence) that there exists a rotation $R\in\SO 3$
such that
\begin{equation}\label{eq:41}
  \nabla v_j \to  R\qquad\text{strongly in }L^2(S^{-1}Z\super k_\rho).
\end{equation}
On the other side, the periodicity of
$z\mapsto \psi_j(Sz)$ in its first variable yields the identity
\begin{equation*}
  \frac{1}{|S^{-1}Z\super k_\rho|}\int_{S^{-1}Z\super
    k_\rho}\partial_1 v_j\ud z= (\id+h_jG)Se_1.
\end{equation*}
By \eqref{eq:41} the left hand side converges to $R e_1$, while the
right hand side converges to $Se_1$; since $|Se_1|\neq 1=|Re_1|$, this is a
contradiction.
\medskip

The argument for \eqref{eq:35} relies on a bending ansatz for
inextensible rods. Let $v_0=(v_0^1,v_0^2,v_0^3)$ denote a smooth, $(1+s)$-periodic curve
from $\Rr$ to $\Rrv 3$ with $|v_0'|\equiv 1$ and
\begin{equation}  
  v_0(0)=(0,\tfrac{1}{2},\tfrac{3}{4}),\quad
  v_0(1+s)=v_0(0)+e_1,\quad v_0'(0)=v_0'(1+s)=e_1.
\end{equation}
Further, we assume that the curve is parallel to the $(x_1,x_2)$-plane,
i.e. $v_0^3\equiv \tfrac{3}{4}$. Let $t_0:=v_0'$ denote the tangent,
$n_0:=-t_0\wedge e_3$ the normal and $\kappa_0:=t_0\cdot n_0'$ the
signed curvature of the curve. Now, define the scaled curve 
\begin{equation*}
  v\super k(y_1):= k v_0(y_1/k)
\end{equation*}
and let $t\super k,n\super k$ and $\kappa\super k$ denote the associated tangent, normal
and curvature. Note that the curvature scales as
\begin{equation}\label{eq:33}
  \kappa\super k(y_1)=\frac{1}{k}\kappa_0(y_1/k).
\end{equation}
Based on $v\super k$ we construct a 3d deformation. For
$y\in Z\super k_\rho$ set
\begin{equation*}
  u\super k_h(y):=v\super k((1+s)y_1)+(y_2-\tfrac{1}{2})n\super k((1+s)y_1)+(y_3-\tfrac{3}{4})e_3+h
  G e_1 y_1.
\end{equation*}
By construction the map $\psi(y):=u\super k_h(y)-(\id+h G)y$ belongs
to $X_\rho\super k$. Thus, by definition of $m\super
k_\rho(\id + h G)$ we have
\begin{equation}\label{eq:34}
  m\super k_\rho(\id +hG)\leq \frac{1}{k}\int_{Z\super k_\rho}W_0(\nabla
  u\super k_h(y)S)\ud y.
\end{equation}
A direct computation shows that
\begin{equation*}
  \nabla u\super k_h=\left(\,t\super k\,|\,n\super k\,|\,e_3\,\right)S^{-1} +
  (y_2-\tfrac{1}{2})\big(\,(n\super k)'{\otimes}e_1\,\big) S^{-1}+hG (e_1{\otimes}e_1).
\end{equation*}
Multiplication with the rotation $R\super k:=(t\super k\,|\,n\super
k\,|\,e_3)$ and a short calculation yields the identity
\begin{equation*}
  (R\super k)\transpose\nabla u\super k_h S =(\id+ \tfrac{1}{k} A\super k+h B\super k)
\end{equation*}
with $A\super k(y):=k\kappa\super
k(y_1)(y_2-\tfrac{1}{2})\,\big(e_1{\otimes}e_1\big)$ and $B\super k:=(R\super
  k)\transpose G\,\big(e_1{\otimes}e_1\big)S$.
Since $|y_2-\tfrac{1}{2}|\leq \rho$ and by appealing to \eqref{eq:33} and the
smoothness of $v_0$, we find that
\begin{equation*}
  \max_{y\in Z\super k_\rho}|A\super k(y)|\leq c\qquad\text{and}\qquad|B\super k|\leq c |Ge_1|
\end{equation*}
where $c$ only depends on the radius $\rho$ and the curve $v_0$. Hence, for small $h$  and large $k$ the matrix $(R\super k)\transpose\nabla u\super k_h(y)
S$ is uniformly close to $\id$ and we can appeal to condition (W4) (and
the frame indifference of $W_0$) to
justify the following expansion: For $\tfrac{1}{k}\leq h$ and $h$ sufficiently
small there holds (uniformly in $y\in Z\super k_\rho$)
\begin{align*}
  W_0(\nabla u\super k_hS)&=W_0((R\super k)\transpose\nabla
  u\super k_hS)=W_0(\id+k^{-1}A\super k+hB\super k)\\
&=Q_0(k^{-1}A\super k+hB\super k)+o(h^2)
\end{align*}
where $Q_0$ denotes the quadratic form associated with the expansion of $W_0$ at
identity. Since $\int_{B_\rho}(y_2-\tfrac{1}{2})\ud (y_2,y_3)=0$ by
symmetry, the components of $A\super k$ and $B\super k$ are orthogonal
in $L^2(B_\rho)$ and we get
\begin{equation*}
  \int_{B_\rho}Q_0(k^{-1}A\super k+hB\super k)\ud (y_2,y_3)=
  \frac{1}{k^2}\int_{B_\rho}Q_0(A\super k)\ud (y_2,y_3)+h^2|B_\rho|\,Q_0(B\super k).
\end{equation*}
In conclusion, we have
\begin{equation*}
  \limsup\limits_{k\to
    \infty}\frac{1}{k}\int_0^k\!\!\int_{B_\rho}W_0(\nabla
  u\super k_hS)\ud y\leq
  \begin{cases}
    0&\text{if }Ge_1=0,\\
    2|B_\rho|h^2\max_{B}Q_0(B)+o(h^2)&\text{else,}
  \end{cases}
\end{equation*}
where the maximum is taken over the set of all cluster points of the bounded sequence
$(B\super k)\subset\Rrm 3$.
In combination with \eqref{eq:34} this implies \eqref{eq:35}.
\medskip

\step 3 (Properties of $I\super k_0$). 
Set
\begin{equation*}
  m\super k_0(F):=\inf\limits_{\varphi\in X\super k_0}I\super k_0(\varphi;F).
\end{equation*}
Let $\mathcal
G:=\{\,G=-g(e_2{\otimes}e_2)\,:\,g\in
  (0,1)\,\}$. We claim that
\begin{align}\label{eq:42}
  \forall G\in\Rrm 3,\,k\in\Nn\,&:\,\limsup\limits_{h\to 0}\frac{1}{h^2}m\super k_0(\id+h G)\leq \frac{1}{2} Q_0(G),\\
  \label{eq:44}\forall G\in\mathcal G,\,k\in\Nn\,&:\,\liminf\limits_{h\downarrow 0}\frac{m\super k_0(\id + h
    G)}{h^2}>0,\\
  \label{eq:43}\forall G\in\mathcal G\,&:\,\limsup\limits_{h\downarrow
    0}\limsup\limits_{k\to\infty}\frac{m\super k_0(\id + h
    G)}{h^2}=0.
\end{align}
Argument for \eqref{eq:42}: By definition of $m\super k_0$ and
condition (W4) we have
\begin{equation*}
  \frac{1}{h^2}m\super k_0(\id + h G)\leq
  \frac{1}{h^2k^2}\int_{Z\super k_0}W_0(\id + h G)\ud x=\frac{1}{2}Q_0(G)+o(h^2)
\end{equation*}
and \eqref{eq:42} follows by passing to the $\limsup_h$ on both sides.

Proof of \eqref{eq:44} by contradiction: Suppose that there exist
$k\in\Nn$, $G\in\mathcal G$ and a vanishing sequence of positive
numbers $(h_j)_{j\in\Nn}$ such that 
\begin{equation}\label{eq:48}
  \lim\limits_{j\to \infty}\frac{1}{h^2_j}m\super k_0(\id+h_jG)= 0.
\end{equation}
By definition of $m\super k_0$ there exists a
sequence $(\varphi_j)\subset X\super k_0$ such that 
\begin{equation*}
  \lim\limits_{j\to\infty}\frac{1}{h^2_j}\int_{Z\super k_0}W_0(\id+h_j(G+\nabla \varphi_j))\ud y= 0.
\end{equation*}
By appealing to the non-degeneracy condition (W3) and geometric
rigidity (see Theorem~\ref{thm:rigidity}) there exist rotations $R_j$
such that
\begin{equation}\label{eq:45}
  \lim\limits_{j\to\infty}h_j^{-1}||\id+h_j(G+\nabla\varphi_j)-R_j||_{L^2(Z\super k_0)}  =0.
\end{equation}
Since $\varphi_j$ is periodic in its first
and second variable, we get for $i=1,2$
\begin{align*}
  h_j^{-1}|(\id-R_j)e_i|\leq &h_j^{-1}|(\id+h_jG-R_j)e_i|+|G
  e_i|\\
  =&h_j^{-1}\left|\,|Z\super k_0|^{-1}\int_{Z\super
      k_0}(\id+h_j(G+\nabla\varphi_j)-R_j)e_i\ud x\,\right|+|Ge_i|\\
  \leq
  &h_j^{-1}||(\id+h_j(G+\nabla\varphi_j)-R_j)e_i||_{L^2(Z\super
    k_0)}|Z\super k_0|^{-1/2} + |Ge_i|
\end{align*}
Hence, in virtue of \eqref{eq:45}, this implies $|(\id -R_j)e_i|\leq c\,h_j$ for $i=1,2$ and a constant $c$
    independent of $j$.
By appealing to the property that $R_j\in\SO 3$, we find that
$|\id-R_j|\leq 2c h_j$. Consequently, by \eqref{eq:45} the sequence $(\nabla\varphi_j)$ is bounded
in $L^2(Z\super k_0)$. Now, we can proceed as in the proof of
Theorem~\ref{thm:1} (Step~2)  and get
\begin{equation*}
  \liminf\limits_{j\to \infty}\frac{1}{h_j^2}m\super k_0(\id + h_jG)\geq
  \inf\limits_{\varphi\in X\super k_0}\frac{1}{k^2}\int_{Z\super
    k_0}Q_0(\sym (G+\nabla\varphi))\ud y.
\end{equation*}
By appealing to the periodicity properties of maps in $X_0\super k$ and since $Q_0$ is positive definite on the subspace of
symmetric matrices, we can easily derive a (non-optimal) lower bound;
namely, the right hand
side is bounded from below by
\begin{equation*}
   c\super k|G:(e_2{\otimes}e_2)|^2\qquad\text{for some positive constant
   }c\super k>0.
\end{equation*}
For $G\in\mathcal G$ this expression is strictly positive --- in
contradiction to \eqref{eq:48}; thus, \eqref{eq:44} follows.
\medskip

The argument for \eqref{eq:43} is more subtle and relies on the
observation that $I\super k_0$ can be written as the elastic energy of a
membrane occupying the slender domain 
\begin{equation*}
  \Omega\super k:=\omega{\times}(0,\tfrac{1}{2k}),\qquad \omega:=(0,1)^2.
\end{equation*}
To make this precise, define the functional
\begin{equation*}
  J\super
  k(u):=\frac{1}{|\Omega\super k|}\int_{\Omega\super k}W_0(\nabla
  u)\ud z,\qquad u\in W^{1,2}(\Omega\super k;\Rrv 3),
\end{equation*}
and notice that with the scaling $\hat u(x)=k\,u(x/k)$ we have
\begin{equation*}
  J\super k(u)=\frac{1}{2k^2}\int_{Z\super k_0}W_0(\nabla \hat u)\ud
  x.
\end{equation*}
 As a consequence it follows that
\begin{multline}\label{eq:46}
  \frac{m_0\super k(\id + h G)}{2}\leq \inf\left\{\,J\super
    k(u)\,:\,\,u\in X\super k_{\id+hG}\right\},\\
\text{where }X\super k_{F}:=\left\{\,u\in W^{1,2}(\Omega\super k;\Rrv 3),\;u(z)=Fz\text{ on
      }\partial\omega{\times}(0,\tfrac{1}{2k})\,\right\}.
\end{multline}
In \cite{Raoult1995} it is shown that $J\super k$ $\Gamma$-converges as
$k\to\infty$ to
a  membrane energy which is zero for short maps, i.e. for deformations $u\,:\,\omega\to\Rrv 3$ with $\nabla u\transpose\nabla
u\leq \id$. Since $(\id+hG)\transpose(\id+hG)\leq \id$ for all
$G\in\mathcal G$ and sufficiently small $h>0$, it is natural to
prove \eqref{eq:43} by appealing to a suitable recovery sequence.
For our purpose it is not necessary to derive the precise
form of the limiting functional. Indeed, the following is sufficient:
For all $G\in\mathcal G$ and $0<h\ll 1$ there exists a sequence
$(u\super k)$, $u\super k\in X\super k_{\id+hG}$ such that
\begin{equation*}
  \limsup\limits_{k\to\infty}J\super k(u\super k)=0.
\end{equation*}
In combination with \eqref{eq:46} this clearly proves the assertion.
The existence of such a sequence is shown for instance in \cite{Conti2003,Raoult1995}.
\medskip

\step 4 (Conclusion). Argument for a). We only prove
\eqref{eq:39}. Since $W_0$ is non-negative and non-degenerate (see
condition (W2)), we have for all $\psi\in W^{1,2}_\per(Y;\Rrv 3)$
\begin{equation*}
\int_Y W(y,F+\nabla\psi)\ud y\geq \int_{Y_0}W_0(F+\nabla\psi)\ud y\geq
c'\int_{Y_0}\dist^2(F+\nabla \psi,\SO 3)\ud y.
\end{equation*}
Here and below, $c',c'',c'''$ are positive constants that only
depend on $W_0$ and on the geometry of $Y_0$.
By geometric rigidity (see
Theorem~\ref{thm:rigidity}) we get (for some $R\in\SO 3$)
\begin{equation*}
\int_Y W(y,F+\nabla\psi)\ud y\geq c'\int_{Y_0}|F+\nabla \psi-R|^2\ud y\stackrel{(\star)}{\geq}
c''|F-R|^2\geq c'''\dist^2(F,\SO 3).
\end{equation*}
Inequality $(\star)$ is valid, since
\begin{equation*}
  \inf_{G,\psi}\int_{Y_0}|G+\nabla\psi(y)|^2\ud y>0
\end{equation*}
where the infimum is taken over all $G\in\Rrm 3$ with $|G|=1$ and $\psi\in W^{1,2}_\per(Y;\Rrv
    3)$ satisfying $\int_0^1\int_0^1\partial_3\psi\ud y_2\ud y_1=0$.

Proof of b) by contradiction. Suppose that there exists $F\in\Rrm 3$
and $k\in\Nn$ such that $W\super k_\hom(F)=0$. Then \eqref{eq:31}
implies that $m\super k_0(F)=m\super k_\rho(F)=0$. As in the proof of
\eqref{eq:35}, we find that $m\super k_0(F)=0$ implies $|Fe_1|=1$, while $m\super k_\rho(F)=0$ implies $|Fe_1|>1$. But
this is a contradiction.

Statement c) is a direct consequence of \eqref{eq:35}, the splitting
in Step~1 and the property that $W(y,\id)=W_0(\id)=0$ for $y\in Y_0$.

Statement d) follows from \eqref{eq:35}, the splitting in Step~1 and \eqref{eq:42}.

Proof of e). Let $G\in\mathcal G$, i.e. $G=-g(e_2{\otimes}e_2)$ where
$g\in(0,1)$. Note that by Step~1 we have $W\super k_\hom(\id)=m\super k_0(\id)+m\super
k_\rho(\id)=m\super k_\rho(\id)$ and therefore
\begin{equation}\label{eq:51}
  W\super k_\hom(\id+h G)-W\super k_\hom(\id)=m\super k_0(\id+ hG).
\end{equation}
With \eqref{eq:42} and \eqref{eq:44} we obtain
\begin{equation*}
  0<\liminf\limits_{h\downarrow 0}\frac{m\super
    k_0(\id + h G)}{h^2}\leq \limsup\limits_{h\downarrow 0}\frac{m\super
    k_0(\id + h G)}{h^2}\leq \frac{1}{2}Q_0(G).
\end{equation*}
Hence, if $\sigma\super k\in\Rrm 3$ fulfills \eqref{eq:47} it must
satisfy $\iprod{\sigma\super k}{G}=0$ and we get
\begin{equation*}
  \liminf\limits_{h\downarrow 0}\frac{1}{h^2}\left(W\super
    k_\hom(\id+hG)-W\super k_\hom(\id)-h\iprod{\sigma\super k}{G}\right)>0.
\end{equation*}
On the other side, the combination of \eqref{eq:51} and \eqref{eq:43}  implies that
\begin{equation*}
  \lim\limits_{h\downarrow 0}\frac{1}{h^2}W\mc(\id+hG)=0,
\end{equation*}
which completes the proof.
 \end{proof}

\section*{Acknowledgements}
We thank S.~Conti for very inspiring discussions. The second author
gratefully acknowledges the hospitality of the Hausdorff Center for
Mathematics in Bonn, as part of this work was accomplished there.

\bibliographystyle{spmpsci}

\end{document}